\RequirePackage[l2tabu, orthodox]{nag}
\documentclass[a4paper, 11pt]{article}

\usepackage[top=1truein, bottom=1truein, left=1truein, right=1truein]{geometry}

\usepackage{type1cm}
\usepackage[T1]{fontenc}
\usepackage{lmodern}

\usepackage{amsmath}
\usepackage{amssymb}
\usepackage{bm}
\usepackage{mathtools}
\usepackage{textcomp}

\usepackage{hyperref}
\usepackage{url}

\usepackage{color}

\usepackage{blkarray}
\usepackage{multirow}
\usepackage{bigdelim}
\usepackage{booktabs}
\renewcommand{\arraystretch}{0.8}

\usepackage{siunitx}

\usepackage[capitalize, noabbrev]{cleveref}
\usepackage{autonum}

\usepackage{graphicx}
\usepackage{tikz}
\usetikzlibrary{arrows.meta}
\usetikzlibrary{patterns}
\usetikzlibrary{positioning}
\usetikzlibrary{decorations.pathmorphing}
\usepackage{circuitikz}
\usepackage{gnuplot-lua-tikz}
\usepackage[subrefformat=parens]{subcaption}

\usepackage{okicmd}
\usepackage{okithm}

\renewcommand\P[1]{\mathrm{P}(#1)}
\newcommand  \D[1]{\mathrm{D}(#1)}
\newcommand \degdet[1]{\delta_n(#1)}
\newcommand\tdegdet[1]{\hat{\delta}_n(#1)}
\newcommand\reor[1]{{#1}^\circ}

\newcommand\LM[1]{{#1}^{\mathrm{LM}}}
\newcommand{\CC}{C\nolinebreak\hspace{-.05em}\raisebox{.4ex}{\tiny\textbf{+}\nolinebreak\hspace{-.10em}\tiny\textbf{+}}}
\renewcommand\setF{\mathbf{F}}
\renewcommand\setK{\mathbf{K}}
\newcommand{\T}{\mathsf{T}}

\newcommand\overmat[3]{%
  \BAmulticolumn{#1}{c}{%
    \overbrace{%
      \hphantom{%
        \begin{matrix}#2\end{matrix}%
      }%
    }^{#3}%
  }%
}
\newcommand\rightmat[2]{\rdelim\}{#1}{0em}[$\scriptstyle{#2}$]}

\usepackage[en-US]{datetime2}

\title{%
  Index Reduction for Differential-Algebraic Equations\\
  with Mixed Matrices%
  \footnote{A preliminary version of this paper is to appear in Proceedings of the Eighth SIAM Workshop on Combinatorial Scientific Computing, Bergen, Norway, June 2018.}
}
\author{Satoru Iwata%
  \thanks{%
    Department of Mathematical Informatics,
    Graduate School of Information Science and Technology,
    The University of Tokyo,
    Hongo 7-3-1, Bunkyo-ku, Tokyo 113-8656, Japan.
    E-mail: \{\href{iwata@mist.i.u-tokyo.ac.jp}{\texttt{iwata}}, %
             \href{taihei_oki@mist.i.u-tokyo.ac.jp}{\texttt{taihei\_oki}}%
           \}\texttt{@mist.i.u-tokyo.ac.jp}%
  }%
  \and%
  Taihei Oki%
  \footnotemark[2]%
  \and%
  Mizuyo Takamatsu%
  \thanks{%
    Department of Information and System Engineering,
    Chuo University,
    Kasuga 1-13-27, Bunkyo-ku, Tokyo 112-8551, Japan.
    E-mail: \url{takamatsu@ise.chuo-u.ac.jp}%
  }
}

\newcommand{\mykeywords}{differential-algebraic equations, index reduction, combinatorial relaxation, matroid theory, combinatorial matrix theory, combinatorial scientific computing}

\hypersetup{
  bookmarksnumbered=true,
  pdfauthor={Satoru Iwata, Taihei Oki and Mizuyo Takamatsu},
  pdfdisplaydoctitle=true,
  pdfkeywords={\mykeywords},
  pdftitle={Index Reduction for Differential-Algebraic Equations with Mixed Matrices},
  setpagesize=false,
}

\begin{document}

\maketitle

\begin{abstract}
  Differential-algebraic equations (DAEs) are widely used for modeling of dynamical systems.
  The difficulty in solving numerically a DAE is measured by its differentiation index.
  For highly accurate simulation of dynamical systems, it is important to convert high-index DAEs into low-index DAEs.
  Most of existing simulation software packages for dynamical systems are equipped with an index-reduction algorithm given by Mattsson and S\"{o}derlind.
  Unfortunately, this algorithm fails if there are numerical cancellations. 
   
  These numerical cancellations are often caused by accurate constants in structural equations.
  Distinguishing those accurate constants from generic parameters that represent physical quantities, Murota and Iri introduced the notion of a mixed matrix as a mathematical tool for faithful model description in structural approach to systems analysis. 
  For DAEs described with the use of mixed matrices, efficient algorithms to compute the index have been developed by exploiting matroid theory. 
  
  This paper presents an index-reduction algorithm for linear DAEs whose coefficient matrices are mixed matrices, i.e., linear DAEs containing physical quantities as parameters.
  Our algorithm detects numerical cancellations between accurate constants, and transforms a DAE into an equivalent DAE to which Mattsson--S\"{o}derlind's index-reduction algorithm is applicable.
  Our algorithm is based on the combinatorial relaxation approach, which is a framework to solve a linear algebraic problem by iteratively relaxing it into an efficiently solvable combinatorial optimization problem.
  The algorithm does not rely on symbolic manipulations but on fast combinatorial algorithms on graphs and matroids.
  Our algorithm is proved to work for any linear DAEs whose coefficient matrices are mixed matrices.
  Furthermore, we provide an improved algorithm under an assumption based on dimensional analysis of dynamical systems. 
  Through numerical experiments, it is confirmed that our algorithms run sufficiently fast for large-scale DAEs, and output DAEs such that physical meanings of coefficients are easy to interpret.
  Our algorithms can also be applied to nonlinear DAEs by regarding nonlinear terms as parameters.
  
\end{abstract}

\section{Introduction}
An $l$th order \emph{differential-algebraic equation} (DAE) for $\funcdoms{x}{\setR}{\setR^n}$ is a differential equation in the form of
\begin{align} \label{def:dae}
  F\prn[\Big]{t, x(t), \dot{x}(t), \ldots, x^{(l)}(t)} = 0,
\end{align}
where $\funcdoms{F}{\setR \times \setR^n \times \cdots \times \setR^n}{\setR^n}$ is a sufficiently smooth function.
DAEs have aspects of both ordinary differential equations (ODEs) $\dot{x}(t) = \phi(t, x(t))$ and algebraic equations $G(t, x(t)) = 0$.
DAEs are widely used for modeling of dynamical systems, such as mechanical systems, electrical circuits and chemical reaction plants.

The difficulty in solving numerically a DAE is measured by its \emph{differentiation index}~\cite{Brenan1996}, which is defined for a first-order DAE
\begin{align} \label{def:1st_dae}
  F\prn[\big]{t, x(t), \dot{x}(t)} = 0
\end{align}
as the smallest nonnegative integer $\nu$ such that the system of equations
\begin{align}
  F\prn[\big]{t, x(t), \dot{x}(t)} = 0, \quad
  \dif{}{t} F\prn[\big]{t, x(t), \dot{x}(t)} = 0, \quad
  \ldots, \quad
  \dif{^\nu}{t^\nu} F\prn[\big]{t, x(t), \dot{x}(t)} = 0
\end{align}
can determine $\dot{x}$ as a continuous function of $t$ and $x$.
That is, $\nu$ is the number of times one has to differentiate the DAE~\eqref{def:1st_dae} to obtain an ODE.
Intuitively, the differentiation index represents how far the DAE is from an ODE.
The differentiation index of an $l$th order DAE~\eqref{def:dae} is defined as that of the first-order DAE obtained by replacing higher-order derivatives of $x$ with newly introduced variables.

A common approach for solving a high $\prn{\geq 2}$ index DAE is to convert it into a low $\prn{\leq 1}$ index DAE. 
This process is called \emph{index reduction}, and it is important for accurate simulation of dynamical systems. 
Most of existing simulation software packages for dynamical systems, such as Dymola, OpenModelica, MapleSim and Simulink, are equipped with the index-reduction algorithm given by Mattsson--S\"{o}derlind~\cite{Mattsson1993} (MS-algorithm).
The MS-algorithm uses Pantelides' method~\cite{Pantelides1988} as a preprocessing step.
Pantelides' method constructs a bipartite graph from structural information of a given DAE and solves an assignment problem on the bipartite graph efficiently.
The MS-algorithm then differentiates equations in the DAE with the aid of the information obtained by Pantelides' method, and replaces some derivatives with dummy variables.
The MS-algorithm returns a sparse DAE if the given DAE is sparse, and thus the algorithm can be applied to large scale DAEs.

Pantelides' method, however, does not work correctly even for the following simple DAE
\begin{align}
  \begin{cases}
    \dot{x}_1 + \dot{x}_2 + x_3 = 0, \\
    \dot{x}_1 + \dot{x}_2 \phantom{{}+x_3} = 0, \\
    \phantom{x_1+{}} x_2 + \dot{x}_3 = 0. \\
  \end{cases}
\end{align}
Pantelides' algorithm reports that the index is zero, whereas it is indeed two.
This is because the method cannot detect the singularity of the coefficient matrix
\begin{align}
  \begin{pmatrix}
    1 & 1 & 0 \\
    1 & 1 & 0 \\
    0 & 0 & 1
  \end{pmatrix}
\end{align}
of $(\dot{x}_1, \dot{x}_2, \dot{x}_3)^\top$.
As this toy example shows, Pantelides' method, which ignores numerical information, may fail on some DAEs due to numerical cancellations.
This kind of failure can also occur in other methods to reduce the index or to analyze DAEs such as the structural algorithm of Unger et al.~\cite{Unger1995} and the $\Sigma$-method of Pryce~\cite{Pryce2001}.

Some index reduction algorithms address this problem.
One example is the $\sigma \nu$-method by Chowdhry et al.~\cite{Chowdhry2004}, which is based on the algorithm by Unger et al.~\cite{Unger1995}.
The method performs Gaussian elimination on the Jacobian matrix $\partial F / \partial \dot{x}$ under the assumption that nonlinear or time-varying terms do not cancel out.
For first-order linear DAEs with constant coefficients, Wu et al.~\cite{Wu2013} proposed a method (WZC-method) to transform a DAE into an equivalent DAE to which the MS-algorithm is applicable.
This method adopts the \emph{combinatorial relaxation} framework introduced by Murota~\cite{Murota1990, Murota1995a} to solve a linear algebraic problem by iteratively relaxing it into an efficiently solvable combinatorial optimization problem.
The combinatorial relaxation based approach was extended for nonlinear DAEs implicitly by Tan et al.~\cite{Tan2017} as the LC-method and the ES-method, and explicitly by~\cite{Oki2019} as the substitution and augmentation methods.
These methods identify numeric or symbolic cancellations and modify the DAE if necessary.

Pantelides' method~\cite{Pantelides1988} and the $\Sigma$-method~\cite{Pryce2001} discard numerical information, which sometimes leads to a failure of the methods. 
In dynamical systems, specific numbers in structural equations, such as in the conservation laws, should be treated as constants, while we can deal with physical characteristic values as nonzero parameters without reference to their values. 
For a faithful model of a dynamical system, it is natural to distinguish accurate and inaccurate numbers. 
This led Murota--Iri~\cite{Murota1985a} to introduce the notion of a \emph{mixed matrix}, which is a matrix consisting of the following two kinds of entries: 
\begin{description}
  \item[Accurate Constants\normalfont{,}]
    which represent precise values such as coefficients of conservation laws.
    We assume that arithmetic operations with these constants can be performed in constant time.
    
  \item[Independent Parameters\normalfont{,}]
    which are algebraically independent over the field of accurate constants.
    These parameters often represent physical quantities such as masses, lengths or electric resistances since their values are inaccurate by measurement noise and other errors.
    These parameters should be treated combinatorially without reference to their values.
\end{description}

For example, consider an electric network consisting of voltage sources, resistances and wires connecting them.
A system of linear equations representing the circuit has two kinds of coefficients: the exact `$\pm 1$'s coming from Kirchhoff's laws, and the resistance values coming from Ohm's law.
Since the values of resistances are usually inaccurate, it is natural to model the system by a linear equation with a mixed matrix, where constants and parameters represent the exact `$\pm 1$'s and the resistances, respectively.
See an example in \cref{sec:mixed_matrices_and_mixed_polynomial_matrices} for modeling of an RLC circuit with a mixed matrix.

Mixed matrices can be handled by symbolic computation systems.
However, the computational cost of symbolic manipulation grows explosively when the size of matrices increases.
Efficient algorithms without symbolic manipulation are available for the rank computation~\cite{Murota1993}.
If all nonzero entries of a matrix are independent parameters, then its rank is equal to the maximum size of a matching in an associated bipartite graph.
For a mixed matrix, the rank computation corresponds to solving an independent matching problem on matroids, which is a generalization of the maximum matching problem on bipartite graphs.
An efficient algorithm based on matroid theory is provided for the rank computation of mixed matrices; see \cite{Murota2000} for details.
Algorithms combining the combinatorial relaxation and mixed matrices are presented in~\cite{Iwata2001, Iwata2013, Sato2015}.
\cref{ref:combinatorial_relaxation_algorithm_for_mixed_polynomial_matrices} describes the history of this combination in detail.

In this paper, we provide an index reduction algorithm for a linear DAE
\begin{align} \label{def:linear_dae}
  \sum_{k=0}^l A_k x^{(k)}(t) = f(t)
\end{align}
with $n \times n$ mixed matrices $A_0, A_1, \ldots, A_l$ and a sufficiently smooth function $\funcdoms{f}{\setR}{\setR^n}$.

A typical procedure to analyze a linear dynamical system with our algorithm is as follows.
First, we model the system by a linear DAE~\eqref{def:linear_dae} with mixed matrices.
Next, we apply our algorithm to the DAE and obtain a low-index one.
We finally obtain a numerical solution by applying a numerical scheme to the low-index DAE after substituting specific values of physical quantities.

As described above, a notable feature of our algorithm is that it works for linear DAEs containing physical quantities not as their values but as parameters.
This feature is advantageous in the following points.
First, since accurate constants arising from typical dynamical systems are integers or rational numbers, our algorithm can avoid arithmetic operations with floating-point numbers.
This fact makes it numerically stable, as index reduction algorithms involve nonsingularity checking of matrices.
Second, since our algorithm can utilize the beforehand knowledge that independent parameters do not cause numerical cancellations, our algorithm is expected to run faster for DAEs with dense coefficient matrices than other index reduction algorithms.
Third, when we simulate a dynamical system on many different values of physical quantities, we can reuse the resulting low-index DAE as long as the values of physical quantities do not unluckily cancel out.

Our algorithm is based on the combinatorial relaxation framework as the WZC-method.
To detect and resolve numerical cancellations in mixed matrices without using symbolic manipulations, we present a new combinatorial relaxation algorithm relying on the theory of combinatorial optimization algorithms on matroids.
Our algorithm is proved to run in $\Order\prn{l^2 n^{\omega + 2}}$ time, where $\omega$ is the matrix multiplication exponent, i.e., the number of arithmetic operations needed to multiply two $n \times n$ matrices is $\Order\prn{n^\omega}$.
The current best known value of $\omega$ is $\approx 2.3728639$ due to~\cite{Gall2014}.
In practice, however, we adopt $\omega = 3$ for performance comparisons because large coefficients are hidden in the big-O notation of the time complexity of fast matrix multiplication algorithms.
Our algorithm is expected to run much faster in most cases because it terminates without modifying the DAE unless it has numerical cancellations.

In addition, we give an improved algorithm for DAEs whose coefficients are \emph{dimensionally consistent}.
The dimensional consistency, which is introduced by Murota~\cite{Murota1985b}, is a mathematical assumption on mixed matrices reflecting the principle of dimensional homogeneity in physical systems.
DAEs arising from dynamical systems naturally ensure this assumption.
We show that the improved algorithm retains the dimensional consistency, and 
that the running time is $\Order\prn{l n^4 \log n}$.
In practice, we confirm through numerical experiments that this algorithm is so fast that it runs in 12 minutes for a DAE with sparse coefficient matrix of size $131{,}076 \times 131{,}076$ having 393,223 nonzero entries.
The running time grows proportionally as $\Order\prn{n^3}$ for dense DAEs and $\Order\prn{n^2}$ for sparse DAEs in our experiments.

It is further confirmed that our algorithm modifies DAEs preserving ``physical meanings'' of dynamical systems.
For example, for a DAE representing the Butterworth filter via the fourth Cauer topology (see equation~\eqref{eq:cauer} in \cref{sec:experiments}), the existing method (LC-method) and our algorithm return
\begin{align} \label{eq:result_DAE}
  \renewcommand{\arraystretch}{0.5}
  \left\{\begin{aligned}
    -\xi_0 + \xi_1 + \xi_2 &= 0, \\
    -\xi_2 + \xi_3 + \xi_4 &= 0, \\
    -\xi_0 + \xi_1 + \xi_3 - 0.243624 \dot{\xi}_4 + 0.318310 \dot{\eta}_3 &= 0, \\
    \eta_1 &= -V(t), \\
    z_1^{[1]} &= -\dot{V}(t),\\
    1.847759 \dot{\xi}_2 - \eta_1 + \eta_3 &= 0, \\
    0.765367 \dot{\xi}_4 + 3.141593 \xi_5 - \eta_3 &= 0, \\
    \eta_0 &= V(t), \\
    -\xi_1 + 0.765367 z_1^{[1]} &= 0, \\
    1.847759 \dot{\xi}_2 - \eta_2 &= 0, \\
    -\xi_3 + 1.847759 \dot{\eta}_3 &= 0, \\
    0.765367 \dot{\xi}_4 - \eta_4 &= 0, \\
    3.141593 \xi_5 - \eta_5 &= 0
  \end{aligned}\right.
  \text{and} \quad
  \left\{\begin{aligned}
    -\xi_0 + \xi_1 + \xi_2 &= 0, \\
    -\xi_2 + \xi_3 + \xi_4 &= 0, \\
    -\xi_0 + \xi_1 + \xi_3 + \xi_5 &= 0, \\
    \eta_0 + \eta_1 &= 0, \\
    z_0^{[1]} + z_1^{[1]} &= 0, \\
    -\eta_1 + \eta_2 + \eta_3 &= 0, \\
    -\eta_3 + \eta_4 + \eta_5 &= 0, \\
    \eta_0 &= V(t), \\
    z_0^{[1]} &= \dot{V}(t), \\
    -\xi_1 + C_1 z_1^{[1]} &= 0, \\
    L_2 \dot{\xi}_2 - \eta_2 &= 0, \\
    -\xi_3 + C_3 \dot{\eta}_3 &= 0, \\
    L_4 \dot{\xi}_4 - \eta_4 &= 0, \\
    R \xi_5 - \eta_5 &= 0,
  \end{aligned}\right.
\end{align}
respectively.
Here $\xi_0, \ldots, \xi_5, \eta_0, \ldots, \eta_5, z_0^{[1]}$ and  $z_1^{[1]}$ are variables of these DAEs, $\funcdoms{V}{\setR}{\setR}$ is a smooth function, and $C_1, C_3, L_2, L_4, R$ are constants representing physical quantities in the circuit.
As the LC-method requires substituting specific values into physical quantities beforehand, the values of physical quantities $C_1 = L_4 \simeq 0.765367$, $L_2 = C_3 \simeq 1.847759$ and $R \simeq 3.141593$ are scattered around the left DAE in~\eqref{eq:result_DAE}, and some coefficients are combinations of them: $1/R \simeq 0.318310$ and $L_4/R \simeq 0.243624$.
This makes it difficult to guess where such complicated coefficients come from and how the DAE changes when the values of physical quantities are varied.
However, our algorithm is directly applicable to DAEs containing physical quantities as parameters.
Thus the right DAE in~\eqref{eq:result_DAE} still retains units of physical quantities, and it is easy to interpret what the coefficients mean in the circuit.

Furthermore, though our index reduction algorithm is designed for linear DAEs, it can be applied to nonlinear DAEs by regarding nonlinear terms as independent parameters.
A similar approach is adopted in the $\sigma \nu$-method of Chowdhry et al.~\cite{Chowdhry2004}, which is adopted in Mathematica~\cite{mathematica}.
However, our method is expected to be applicable to a larger class of nonlinear DAEs than the $\sigma \nu$-method because our method does not transform a DAE involving nonlinear terms.
Indeed, consider the index-2 nonlinear DAE
\begin{align} \label{def:nonlinear_dae}
  \left\{\begin{alignedat}{4}
    F_1:\:\: & \dot{x}_1 + g(x_2)  &&= f_1(t), \\
    F_2:\:\: & \dot{x}_1 + x_1 + x_3 &&= f_2(t), \\
    F_3:\:\: & \dot{x}_1 \phantom{{}+x_1} + x_3 &&= f_3(t)
  \end{alignedat}\right.
\end{align}
with smooth functions $\funcdoms{f_1, f_2, f_3, g}{\setR}{\setR}$.
Our algorithm correctly returns an index-1 DAE equivalent to~\eqref{def:nonlinear_dae}, whereas the implementation of the $\sigma \nu$-method in Mathematica unsuccessfully returns an index-2 DAE.
See \cref{sec:nonlinear_dae} for details.

\paragraph{Related work.}

We describe the relation between the proposed algorithm and related index reduction algorithms.
If all nonzero entries of $A(s)$ are independent parameters, our algorithm just passes a given DAE to the MS-method.
In contrast, if $A(s)$ has no independent parameters, then our algorithm coincides with the LC-method by Tan et al.~\cite{Tan2017} and with the substitution method~\cite{Oki2019} applied to linear DAEs with constant coefficients.
We emphasize that our algorithm can treat intermediate DAEs between these special cases, i.e., it works for DAEs containing both accurate constants and independent parameters.

The WZC-method by Wu et al.~\cite{Wu2013} works for first-order linear DAEs with constant coefficients.
This method modifies a DAE using the combinatorial relaxation method in~\cite{Iwata2003}, which performs row and column operations on $A(s)$ using constant matrices.
Here, column operations on $A(s)$ correspond to changing the basis of the variable space of DAEs.
Our combinatorial relaxation algorithm does not use column operations, and thus the basis of the variable space remains unchanged.

A recent work~\cite{Iwata2018a} has proposed an index reduction algorithm which is proved to work for any instances of first order linear DAEs with constant coefficients.
The algorithm directly reduces the index of a given DAE by row operations, whereas our algorithm only resolves numerical cancellations in a DAE and eventually relies on the MS-algorithm for the actual index reduction process.
Thus our algorithm is expected to preserve the sparsity of DAEs compared to the algorithm in~\cite{Iwata2018a}.

In addition, our algorithm is similar to the $\sigma \nu$-method~\cite{Chowdhry2004} in the sense that both methods treat matrices having accurate constants and independent parameters, yet their approaches are quite different; the $\sigma \nu$-method is based on the Gaussian elimination approach by Gear~\cite{Gear1988}, whereas our algorithm relies on the dummy variable approach by Mattsson--S\"{o}derlind~\cite{Mattsson1993}.

\paragraph{Organization.}

The rest of this paper is organized as follows.
\Cref{sec:index_reduction_for_linear_DAEs} reviews the previous index computation and reduction algorithms for linear DAEs with constant coefficients, including the MS-algorithm and combinatorial relaxation algorithms.
\Cref{sec:mixed_matrices} explains mixed matrices and their rank identities.
\Cref{sec:algorithm} describes the proposed algorithm.
\Cref{sec:exploiting_dimensional_consistency} improves our algorithm under the assumption of the dimensional consistency.
\Cref{sec:example} illustrates the theory by two examples.
\Cref{sec:experiments} shows the result of numerical experiments.
\Cref{sec:nonlinear_dae} discusses an application to nonlinear DAEs.
Finally, \cref{sec:conclusion} concludes this paper.

\section{Index Reduction for Linear DAEs}
\label{sec:index_reduction_for_linear_DAEs}

\subsection{Index of Linear DAEs}

A linear DAE with constant coefficients is
\begin{align} \label{def:linear_dae2}
  \sum_{k=0}^l A_k x^{(k)}(t)= f(t),
\end{align}
where $A_0, A_1, \ldots, A_l$ are $n \times n$ matrices and $\funcdoms{f}{\setR}{\setR^n}$ is a sufficiently smooth function.
We assume that $f$ is Laplace transformable for simplicity, though this assumption is not essential.
By the Laplace transformation, the DAE~\eqref{def:linear_dae2} is transformed into
\begin{align} \label{eq:laplaced_dae2}
  A(s) \widetilde{x}(s)
  = \widetilde{f}(s) + \sum_{k=0}^l \sum_{i=1}^k s^{k-i} A_k x^{(i-1)}(0),
\end{align}
where $\widetilde{x}(s)$ and $\tilde{f}(s)$ are the Laplace transforms of $x(t)$ and $f(t)$, respectively, and $A(s) = \sum_{k=0}^l s^k A_k$.
We henceforth denote the right-hand side of~\eqref{eq:laplaced_dae2} by $\hat{f}(s)$.
The matrix $A(s)$ is a matrix whose entries are polynomials, called a \emph{polynomial matrix}.
We say that $A(s)$ is \emph{nonsingular} if its determinant is not identically zero.

An initial value $\prn[\big]{x_0, x^{(1)}_0, \ldots, x^{(l-1)}_0} \in \setR^n \times \cdots \times \setR^n$ is said to be \emph{consistent} if there exists at least one solution of~\eqref{def:linear_dae2} satisfying
\begin{align} \label{cond:initial_value}
  x(0) = x_0, \, \dot{x}(0) = x^{(1)}_0, \, \ldots, \, x^{(l-1)}(0) = x^{(l-1)}_0.
\end{align}
We say that the DAE~\eqref{def:linear_dae2} is \emph{solvable} if there exists a unique solution of~\eqref{def:linear_dae2} satisfying the initial value condition~\eqref{cond:initial_value} for an arbitrary consistent point.
The solvability of~\eqref{def:linear_dae2} is characterized by $A(s)$ as follows.

\begin{theorem}[{\cite{Brenan1996, Shi2004}}]
  A linear DAE~\eqref{def:linear_dae2} is solvable if and only if the associated polynomial matrix $A(s)$ is nonsingular.
\end{theorem}

See~\cite[Theorem~2.3.1]{Brenan1996} for $l=1$ and~\cite[Theorems~2.22--23]{Shi2004} for $l \geq 2$.
In this paper, we focus on solvable DAEs~\eqref{def:linear_dae2}.
With slight abuse of terminology, we also refer to equation~\eqref{eq:laplaced_dae2} as a DAE.

The differentiation index of the first-order linear DAE~\eqref{def:linear_dae2} with $A(s) = A_0 + sA_1$ is known to be
\begin{align} \label{def:nu_A}
  \nu(A) = \delta_{n-1}(A) - \degdet{A} + 1
\end{align}
as described in~\cite[Remark~5.1.10]{Murota2000}.
Here, $\delta_k(A)$ denotes the maximum degree of the determinant of a submatrix in $A(s)$ of size $k$, i.e.,
\begin{align} 
  \delta_k(A) = \max \set[\big]{\deg \det A(s)[I, J]}[\card{I} = \card{J} = k],
\end{align}
where $A(s)[I, J]$ is the submatrix in $A(s)$ with row set $I$ and column set $J$, and $\deg p(s)$ designates the degree of a polynomial $p(s)$ in $s$.
In particular, $\degdet{A}$ is the degree of the determinant of $A(s)$, and  $\delta_{n-1}(A)$ is the maximum degree of a cofactor of $A(s)$.
For a DAE~\eqref{def:linear_dae2} with $l \geq 2$, its index is defined to be that of the first order DAE obtained by replacing higher-order derivatives with new variables~\cite{Tan2017}.

\subsection{Assignment Problem}

In analysis of DAEs, Pryce~\cite{Pryce2001} introduced an assignment problem as a reinterpretation of Pantelides' algorithm~\cite{Pantelides1988}.
We describe it specializing to linear DAEs~\eqref{def:linear_dae2} using our notations.

Consider a linear DAE~\eqref{eq:laplaced_dae2} with $n \times n$ nonsingular polynomial matrix $A(s)$ with row set $R$ and column set $C$.
We denote the $\prn{i, j}$ entry of $A(s)$ by $A_{i,j}(s)$.
Let $G(A)$ denote the bipartite graph with vertex set $R \cup C$ and edge set $E(A) = \set{\prn{i, j} \in R \times C}[A_{i,j}(s) \neq 0]$.
An edge subset $M \subseteq E(A)$ is called a \emph{matching} if the ends of edges in $M$ are disjoint.
Since $A(s)$ is nonsingular, $G(A)$ has a matching of size $n$, called a \emph{perfect matching}.
We set the weight $c_e$ of an edge $e = \prn{i, j} \in E(A)$ by $c_e = c_{i,j} = \deg A_{i,j}(s)$.

The assignment problem on $G(A)$ is the following problem $\P{A}$:
\Maximize[name=$\P{A}$]{%
  \sum_{e \in M} c_e%
}{%
  \text{$M \subseteq E(F)$ is a perfect matching on $G(F)$.}%
}

The dual problem $\D{F}$ of $\P{F}$ is expressed as follows:
\Minimize[name=$\D{A}$]{%
  \sum_{j \in C} q_j - \sum_{i \in R} p_i%
}{%
  q_j - p_i \geq c_{i,j} & \prn{(i, j) \in E(F)}, \\
  p_i \in \setZ          & \prn{i \in R}, \\
  q_j \in \setZ          & \prn{j \in C}.
}

The integral constraints on $p_i$ and $q_j$ are crucial for analysis of DAEs.
We denote the optimal value of the problem $\P{A}$ (and $\D{A}$) by $\tdegdet{A}$.
Recall that $\degdet{A}$ denotes $\deg \det A(s)$. 
It is well-known that $\degdet{A} \leq \tdegdet{A}$ holds, and the equality is attained if and only if the coefficient of $s^{\tdegdet{A}}$ in $\det A(s)$ does not vanish; see~\cite[Theorem~6.2.2]{Murota2000}.
In this sense, $\tdegdet{A}$ serves as a combinatorial upper bound on $\degdet{A}$.
We call $A(s)$ \emph{upper-tight} if $\degdet{A} = \tdegdet{A}$ holds.

For a dual feasible solution $\prn{p, q}$, a \emph{tight coefficient matrix} $A^\#$ of $A(s)$ is defined by
\begin{align} \label{def:tcf}
  A_{i,j}^\# \defeq \text{the coefficient of $s^{q_j - p_i}$ in $A_{i,j}(s)$}
\end{align}
for each $i \in R$ and $j \in C$.
Note that $A^\#$ changes depending on $\prn{p, q}$.
This matrix is called a ``system Jacobian matrix'' by Pryce~\cite{Pryce2001}; the name ``tight coefficient matrix'' is due to Murota~\cite{Murota1995b}.

\subsection{Computing the Index via Combinatorial Relaxation}
\label{sec:computing_the_index_via_the_combonatorial_relaxation}

The tight coefficient matrix plays an important role in the combinatorial relaxation algorithm of Murota~\cite{Murota1995a} to compute $\delta_n(A)$ for a polynomial matrix $A(s)$ through the following lemma.

\begin{lemma}[{\cite[Proposition 6.2]{Murota1990}}] \label{lem:tightness1}
  Let $A(s)$ be a nonsingular polynomial matrix and let $A^\#$ be the tight coefficient matrix of $A(s)$ with respect to an optimal solution of $\D{A}$.
  Then $A(s)$ is upper-tight if and only if $A^\#$ is nonsingular.
\end{lemma}

The combinatorial relaxation method for computing $\delta_n(A)$ consists of the following three phases.

\begin{enumerate}[label={Phase \arabic*.}]
  \item Compute a combinatorial upper bound $\hat{\delta}_{n}(A)$ of $\delta_{n}(A)$ by solving an assignment problem.
  \item Check whether $A(s)$ is upper-tight using \cref{lem:tightness1}.
        If it is, return $\hat{\delta}_n(A)$ and halt.
  \item Modify $A(s)$ to improve $\hat{\delta}_{n}(A)$ by replacing $A(s)$ with $U(s)A(s)$, where $U(s)$ is a \emph{unimodular matrix}.
        Go back to Phase 2.
\end{enumerate}

Here, a unimodular matrix is a square polynomial matrix whose determinant is a nonzero constant.
The algorithm is designed so that $\tdegdet{A}$ decreases in each iteration, while unimodular transformations preserve $\degdet{A}$.
Thus, after a finite number of iterations, it terminates with $\tdegdet{A} = \degdet{A}$.

Subsequently, Murota~\cite{Murota1995b} applied the combinatorial relaxation approach to computing $\delta_k(A)$ for $k = 1, \ldots, n$.
In this algorithm, Phase~3 modifies $A(s)$ to $U(s)A(s)V(s)$, where $U(s)$ and $V(s)$ are \emph{biproper Laurent polynomial matrices}, i.e., entries are all polynomials in $1/s$ and the determinants are nonzero constants. 
This type of transformation is known to preserve $\delta_{k}(A)$.
The values of $\delta_{n-1}(A)$ and $\delta_n(A)$ determine the index $\nu(A)$ by~\eqref{def:nu_A}.

\subsection{Mattsson--S\"{o}derlind's Index Reduction Algorithm}
\label{sec:ms_alg}

We now review Mattsson--S\"{o}derlind's index reduction algorithm (MS-algorithm) applied to a linear DAE~\eqref{eq:laplaced_dae2} with $n \times n$ nonsingular polynomial matrix $A(s)$.
We remark that the MS-algorithm can be embedded in the $\Sigma$-method of Pryce~\cite{Pryce2001} and they are based on the same principle.

Let $\prn{p, q}$ be an optimal solution of $\D{A}$.
For $h \in \setZ$, we define
\begin{align}
  R_h \defeq \set{i \in R}[p_i = h], &\quad R_{\geq h} \defeq \set{i \in R}[p_i \geq h], \\
  C_h \defeq \set{j \in C}[q_j = h], &\quad C_{\geq h} \defeq \set{j \in C}[q_j \geq h].
\end{align}
The MS-algorithm applied to the DAE~\eqref{eq:laplaced_dae2} is outlined as follows.
The following description is a version specialized to linear DAEs, though the original MS-algorithm is designed for nonlinear DAEs~\cite[Section 3.1]{Mattsson1993}.

\paragraph{Mattsson--S\"{o}derlind's Index Reduction Algorithm}
\begin{enumerate}[label={Step \arabic*.}]
  \item Compute an optimal solution $\prn{p, q}$ of $\D{A}$ satisfying $p_i, q_j \geq 0$ for $i \in R$ and $j \in C$.
        Let $A^\#$ denote the tight coefficient matrix of $A(s)$ with respect to $\prn{p, q}$.
        If $A^\#$ is singular, then the algorithm terminates in failure.
  \item For each $h = 0, \ldots, \eta + 1 \,\, \prn[\big]{\eta \defeq \displaystyle \max_{i \in R} p_i}$, obtain $J_h \subseteq C_{\geq h}$ such that $A^\#[R_{\geq h}, J_h]$ is nonsingular and
        \begin{align}
          C = J_0 \supseteq J_1 \supseteq J_2 \supseteq \cdots \supseteq J_{\eta} \supseteq J_{\eta + 1} = \varnothing.
        \end{align}
  \item For each $j \in C$, let $k_j$ be the integer such that $j \in J_{k_j}$ and $j \notin J_{k_j + 1}$.
        Introduce $k_j$ dummy variables $z^{[q_j]}_j, z^{[q_j-1]}_j, \ldots, z^{[q_j-k_j+1]}_j$ corresponding to $s^{q_j}\tilde{x}_j, \, s^{q_j-1}\tilde{x}_j, \ldots, s^{q_j-k_j+1}\tilde{x}_j$, respectively.
  \item For each $i \in R$, return the 0-th, 1-st, ..., $p_i$-th order derivatives of the $i$-th equation.
        Replace variables with the corresponding dummy variables.
\end{enumerate}

The number of dummy variables introduced in the Step~3 is $\sum_{j \in C} q_j - \hat{\delta}_n(A) = \sum_{i \in R} p_i$, which is equal to the number of the differentiated equations in Step~4.
Since an optimal solution of $\D{A}$ is not unique, the number of dummy variables and equations are not uniquely determined by $A(s)$.
To minimize the the number of equations, Pryce~\cite{Pryce2001} uses the \emph{smallest} optimal solution $(p, q)$ of $\D{A}$, that is, $0 \leq p_i \leq p'_i$ and $0 \leq q_j \leq q'_j$ hold for all nonnegative optimal solution $(p', q')$ of $\D{A}$ and $i \in R$, $j \in C$.
An algorithm to obtain the small $(p, q)$ is known~\cite{Pryce2001}, but there is no guarantee of computational time.

The validity of the MS-algorithm is established as follows.

\begin{proposition}[{\cite[Section 3.2]{Mattsson1993}}] \label{prop:ms-alg}
  Let $A(s)$ be a polynomial matrix in the DAE~\eqref{eq:laplaced_dae2} and $A^\#$ the tight coefficient matrix of $A(s)$ with respect to an optimal solution of $\D{A}$.
  If $A^\#$ is nonsingular, then the MS-algorithm correctly returns an equivalent DAE with index at most one.
\end{proposition}

From \cref{lem:tightness1}, the condition in \cref{prop:ms-alg} is equivalent to the upper-tightness of $A(s)$ as follows.

\begin{corollary} \label{cor:MS_equivalent_validity_condition}
  Let $A(s)$ be a polynomial matrix in the DAE~\eqref{eq:laplaced_dae2}.
  If $A(s)$ is upper-tight, then the MS-algorithm correctly returns an equivalent DAE with index at most one. 
\end{corollary}

The description above is still valid for a nonlinear DAE~\eqref{def:dae} by redefining $A(s)$ as
\begin{align}
  \text{the coefficient of $s^k$ in $A_{i,j}(s)$} \defeq \text{the partial derivative of the $i$-th equation with respect to $x^{(k)}_j$}
\end{align}
for each $i = 1, \ldots, n$, $j = 1, \ldots, n$ and $k = 0, \ldots, l$.
Then the nonsingularity of $A^\#$ essentially comes from the requirement of the implicit function theorem, which is used to convert the DAE into an ODE by solving the DAE for the highest order derivatives.

\subsection{Index Reduction via Combinatorial Relaxation}
\label{sec:index_reduction_algorithm}

For a linear DAE~\eqref{eq:laplaced_dae2} that does not satisfy the validity condition of the MS-algorithm, we need to modify it to apply the MS-algorithm.
Here, the modification of DAEs must preserve the sets of their solutions.
We can use unimodular transformations in the form of
\begin{align} 
  U(s) A(s) \tilde{x}(s) = U(s)\hat{f}(s),
\end{align}
where $U(s)$ is a unimodular matrix.
Since unimodular transformations correspond to the operations of adding an equation or its (higher order) derivative to another equation, the DAEs before and after the transformation have the same solution set.

Murota's combinatorial relaxation algorithm~\cite{Murota1995a} for computing $\delta_n(A)$ described in \cref{sec:computing_the_index_via_the_combonatorial_relaxation} modifies a given polynomial matrix $A(s)$ into an upper-tight polynomial matrix $\bar{A}(s) = U(s)A(s)$ using some unimodular matrix $U(s)$.
Then from \cref{cor:MS_equivalent_validity_condition}, the matrix $\bar{A}(s)$ satisfies the validity condition of the MS-algorithm.
Therefore, we can use Murota's algorithm as an index reduction algorithm by combining it with the MS-algorithm.
Note that this modification may change (increase or decrease) $\delta_{n-1}(A)$, and hence $\nu(A)$.
This method indeed coincides with the LC-method of Tan et al.~\cite{Tan2017} applied to the linear DAEs with constant coefficients.

The idea of using the combinatorial relaxation method as a preprocessing of the MS-algorithm was originally given by Wu et al.~\cite{Wu2013} for first order linear DAEs with constant coefficients.
They proposed the WZC-algorithm that modifies a DAE using the combinatorial relaxation algorithm in~\cite{Iwata2003} for a \emph{matrix pencil} $A(s) = A_0 + sA_1$.
The algorithm in~\cite{Iwata2003} modifies the matrix pencil $A(s)$ to $UA(s)V$, where $U$ and $V$ are nonsingular constant matrices.
Since nonsingular constant matrices are biproper, the values of $\delta_{n-1}(A)$ and $\nu(A)$ do not change in the WZC-algorithm.

\section{DAEs with Mixed Matrices}
\label{sec:mixed_matrices}

The algorithms explained in \cref{sec:index_reduction_for_linear_DAEs} work under the assumption that we know all the values of physical quantities.
In order to treat them as parameters, we deal with a DAE with mixed matrices.

\subsection{Mixed Matrices and Mixed Polynomial Matrices}
\label{sec:mixed_matrices_and_mixed_polynomial_matrices}

Let $\setF$ be a field and $\setK$ a subfield of $\setF$.
A typical setting in the context of DAEs is $\setK = \setQ$ and $\setF$ is the extension field of $\setQ$ obtained by adjoining the set of independent physical parameters.
A matrix $T$ over $\setF$ is said to be \emph{generic} if the set of nonzero entries of $T$ is algebraically independent over $\setK$.
A \emph{mixed matrix} with respect to $\prn{\setK, \setF}$ is a matrix in the form of $Q+T$, where $Q$ is a matrix over $\setK$ and $T$ is a generic matrix.
A mixed matrix $A = Q + T$ is called a \emph{layered mixed matrix} (or \emph{LM-matrix}) if there exists a bipartition $\set{R_Q, R_T}$ of $\Row(A)$ such that all nonzero entries of $Q$ and $T$ are in rows $R_Q$ and $R_T$, respectively.
An LM-matrix $A$ can be expressed as $A = \binom{Q}{T}$.

A polynomial matrix $A(s) = \sum_{k=0}^l s^k A_k$ is called a \emph{mixed polynomial matrix} if it is expressed as $A_k = Q_k + T_k$ with $Q_k$ and $T_k$ that satisfy the following conditions:
\begin{enumerate}[label={(MP-Q)}]
  \item[(MP-Q)] Each $Q_k$ $\prn{k = 0, \ldots, l}$ is a matrix over $\setK$.
  \item[(MP-T)] The set of nonzero entries of $T_0, \ldots, T_l$ is algebraically independent over $\setK$.
\end{enumerate}
A \emph{layered mixed polynomial matrix} (or \emph{LM-polynomial matrix}) is a mixed polynomial matrix such that nonzero rows of $Q(s) = \sum_{k=0}^l s^k Q_k$ and $T(s) = \sum_{k=0}^l s^k T_k$ are disjoint.
An LM-polynomial matrix is expressed as $A(s) = \binom{Q(s)}{T(s)}$.

\begin{figure}[tbp]
  \centering
  \begin{circuitikz}
    \draw (0,-2)
      to[short]  (0,-1)
      to[sV] (0,1)
      to[short, i=$\xi_5$] (0,2)
      to[short]  (2,2)
      to[resistor=$R_1$, i=$\xi_1$, *-*] (2,0)
      to[resistor=$R_2$, i=$\xi_2$, -*] (4,0);
    \draw[-{Triangle[angle=50:1.9mm]}] (-0.7, -0.5) -- (-0.7, 0.5) node[left, midway] {$\eta_5 = V(t)$};
    \draw (-0.9, -0.5) -- (-0.5, -0.5) {};
    \draw (-0.9, 0.5) -- (-0.5, 0.5) {};
    \draw[-{Triangle[angle=50:1.9mm]}] (1.6, 0.4) -- (1.6, 1.6) node[left, midway] {$\eta_1$};
    \draw (1.4, 0.4) -- (1.8, 0.4) {};
    \draw (1.4, 1.6) -- (1.8, 1.6) {};
    \draw[-{Triangle[angle=50:1.9mm]}] (3.6, -0.4) -- (2.4, -0.4) node[below, midway] {$\eta_2$};
    \draw (3.6, -0.2) -- (3.6, -0.6) {};
    \draw (2.4, -0.2) -- (2.4, -0.6) {};
    \draw (2,2)
      to[short] (4,2)
      to[C, l_=$C$, i=$\xi_4$] (4,0)
      to[short] (4,-2)
      to[short] (2,-2);
    \draw[-{Triangle[angle=50:1.9mm]}] (4.6, 0.6) -- (4.6, 1.4) node[right, midway] {$\eta_4$};
    \draw (4.4, 0.6) -- (4.8, 0.6) {};
    \draw (4.4, 1.4) -- (4.8, 1.4) {};
    \draw (2,0)
      to[L=$L$, i=$\xi_3$, -*] (2,-2)
      to[short] (0,-2);
    \draw[-{Triangle[angle=50:1.9mm]}] (1.6, -1.5) -- (1.6, -0.5) node[left, midway] {$\eta_3$};
    \draw (1.4, -1.5) -- (1.8, -1.5) {};
    \draw (1.4, -0.5) -- (1.8, -0.5) {};
  \end{circuitikz}
  \caption{%
    Simple RLC network.
  }
  \label{fig:RLC}
\end{figure}

\begin{example}
  Consider an electrical network illustrated in \cref{fig:RLC}, given in~\cite[Section 1.1]{Murota2000}.
  The network consists of a voltage source of time-varying voltage $V(t)$, two resistances $R_1$ and $R_2$, an inductor $L$ and a capacitor $C$.
  State variables of this network is currents $\xi_1, \ldots, \xi_5$ and voltages $\eta_1, \ldots, \eta_5$ shown in \cref{fig:RLC}.
  The Laplace transform of an index-2 DAE representing this network is given by
  \begin{align} \label{eq:example2_eq}
    A(s) \tilde{x}(s)
    =
    \prn{\begin{array}{rrrrr|rrrrr}
    -1  &     &     & -1 &  1 &    &    &    &    &    \\
        &   1 &   1 &  1 & -1 &    &    &    &    &    \\
    \hline
        &     &     &    &    &  1 &    &  1 &    & -1 \\
        &     &     &    &    & -1 & -1 &    &  1 &    \\
        &     &     &    &    &    &  1 & -1 &    &    \\
    \hline
    R_1 &     &     &    &    & -1 &    &    &    &    \\
        & R_2 &     &    &    &    & -1 &    &    &    \\
        &     &  sL &    &    &    &    & -1 &    &    \\
        &     &     & -1 &    &    &    &    & sC &    \\
        &     &     &    &    &    &    &    &    &  1  
    \end{array}}
    \begin{pmatrix}
    \tilde{\xi}_1(s) \\
    \tilde{\xi}_2(s) \\
    \tilde{\xi}_3(s) \\
    \tilde{\xi}_4(s) \\
    \tilde{\xi}_5(s) \\
    \tilde{\eta}_1(s) \\
    \tilde{\eta}_2(s) \\
    \tilde{\eta}_3(s) \\
    \tilde{\eta}_4(s) \\
    \tilde{\eta}_5(s)
    \end{pmatrix}
    =
    \begin{pmatrix}
    0 \\
    0 \\
    0 \\
    0 \\
    0 \\
    0 \\
    0 \\
    0 \\
    0 \\
    \tilde{V}(s)
    \end{pmatrix},
  \end{align}
  where empty cells in the coefficient matrix $A(s)$ indicate zero.
  Here, $\tilde{x} = \prn[\big]{\tilde{\xi}_1, \ldots, \tilde{\xi}_5,\tilde{\eta}_1, \ldots, \tilde{\eta}_5}^\top$ is the Laplace transform of the vector $\prn[\big]{\xi_1, \ldots, \xi_5, \eta_1, \ldots, \eta_5}^\top$ of variables and $\tilde{V}(s)$ is the Laplace transform of $V(t)$ (we assumed that all state variables and their derivatives were equal to zero at $t = 0$ for simplicity).
  In this system~\eqref{eq:example2_eq}, the first two equations come from Kirchhoff's current law (KCL), and the following three equations come from Kirchhoff's voltage law (KVL).
  The last five equations represent the element characteristics (constitutive equations).
  The coefficient matrix in~\eqref{eq:example2_eq} is naturally regarded as a mixed polynomial matrix with independent parameters $R_1$, $R_2$, $L$ and $C$ since values of the parameters are supposed to be inaccurate.
\end{example}

\subsection{Rank of LM-matrices}
For a matrix $A$, we denote the row set and column set by $\Row(A)$ and $\Col(A)$, respectively.
Consider the associated bipartite graph $G(A) = \prn{R, C; E(A)}$, where $R = \Row(A)$ and $C = \Col(A)$.
The \emph{term-rank} of $A$ is the maximum size of a matching in $G(A)$, and is denoted by $\trank A$.
It is well known that $\rank A \leq \trank A$ holds.
The equality is attained if and only if $A$ has a submatrix of size $\trank A$ with nonzero determinant.
This is analogous to the relation between $\delta_n$ and $\hat{\delta}_n$ for a polynomial matrix.

Let $A = \binom{Q}{T}$ be an LM-matrix.
If $A$ has no accurate constants, i.e., $A$ is a generic matrix $T$, it holds that $\rank T = \trank T$ from the independence of nonzero entries.
From this equality, we can compute $\rank T$ by solving a maximum matching problem on the associated bipartite graph $G(T)$.
For general LM-matrices, the following holds from the generalized Laplace expansion.
\begin{proposition}[{\cite[Theorem 3.1]{Murota1987}}]
  For an LM-matrix $A = \binom{Q}{T}$ with $R_Q = \Row(Q)$, $R_T = \Row(T)$ and $C = \Col(A)$, the following rank identity holds:
  \begin{align} 
    \label{eq:rank_identity_max} \rank A &= \max \set{\rank Q[R_Q, J] + \trank T[R_T, C \setminus J]}[J \subseteq C].
  \end{align}
\end{proposition}

The problem of maximizing the right-hand side of~\eqref{eq:rank_identity_max} can be reduced to an \emph{independent matching problem} on a matroid; see~\cite[Section~4.2]{Murota2000} for details.
The following identity is obtained from the duality of the independent matching problem.
\begin{proposition}[{\cite[Theorem 3.1]{Murota1987}}]
  For an LM-matrix $A = \binom{Q}{T}$ with $R_Q = \Row(Q)$, $R_T = \Row(T)$ and $C = \Col(A)$, the following rank identity holds:
  \begin{align} 
    \label{eq:rank_identity_min} \rank A &= \min \set{\rank Q[R_Q, J] + \trank T[R_T, J] + \card{C \setminus J}}[J \subseteq C].
  \end{align}
\end{proposition}

Similarly, we give the following term-rank identity for LM-matrices, which will be used later in the proof of \cref{lem:tilde_A_optimality}.

\begin{proposition} \label{prop:trank_identity_min}
  For an LM-matrix $A = \binom{Q}{T}$ with $R_Q = \Row(Q)$, $R_T = \Row(T)$ and $C = \Col(A)$, the following term-rank identity holds:
  \begin{align} 
    \trank A &= \min \set{\trank Q[R_Q, J] + \trank T[R_T, J] + \card{C \setminus J}}[J \subseteq C].
  \end{align}
\end{proposition}
\begin{proof}
  This immediately follows from the well-known rank formula of a union matroid~\cite{Edmonds1968} and the fact that the union of transversal matroids is also a transversal matroid~\cite[Corollary~11.3.8]{Oxley2011}.
\end{proof}

\subsection{Combinatorial Relaxation Algorithm for Mixed Polynomial Matrices}
\label{ref:combinatorial_relaxation_algorithm_for_mixed_polynomial_matrices}

Murota~\cite{Murota1998} described the first algorithm to compute $\delta_k$ of a mixed polynomial matrix through a reduction to a \emph{valuated independent assignment problem}.
The valuated independent assignment problem is an optimization problem on \emph{valuated matroids}, which are a generalization of matroids.
The mixed matrices concept and the combinatorial relaxation were first combined in~\cite{Iwata2001}.
The algorithm in~\cite{Iwata2001} computes $\delta_k$ of a (usual) polynomial matrix $A(s)$ obtained by plugging in specific values for independent parameters in a mixed polynomial matrix.
Based on the framework of combinatorial relaxation, this algorithm iteratively computes $\delta_k$ of mixed polynomial matrices using~\cite{Murota1998} as a combinatorial upper bound on $\delta_k(A)$.
Subsequently, \cite{Iwata2013} proposed a combinatorial relaxation algorithm for computing $\delta_k$ of mixed polynomial matrices without using valuated matroid theory.
Sato~\cite{Sato2015} presented a fast algorithm to compute the entire sequence $\delta_1, \ldots, \delta_n$ of mixed polynomial matrices extending the algorithm in~\cite{Iwata2013}.

The algorithm in~\cite{Iwata2013} first converts a mixed polynomial matrix into an LM-polynomial matrix $A(s) = \binom{Q(s)}{T(s)}$, and modifies $A(s)$ to
\begin{align}
  \bar{A}(s) =
  \begin{pmatrix} U_Q(s) & O \\ O & I \end{pmatrix}
  \binom{Q(s)}{T(s)},
\end{align}
where $I$ is an identity matrix of appropriate size, and $U_Q(s)$ is a nonsingular \emph{Laurent polynomial matrix}.
Here, a \emph{Laurent polynomial matrix} is a matrix whose entries are polynomials in $s$ and $s^{-1}$.
With the use of~\eqref{def:nu_A}, we can obtain the index $\nu(A)$ by computing $\delta_n(A)$ and $\delta_{n-1}(A)$. 

In order to devise an index reduction algorithm for DAEs with mixed matrices, we need to make use of unimodular transformations instead of Laurent polynomial transformations, as explained in \cref{sec:index_reduction_algorithm}.

\section[%
  Combinatorial Relaxation Algorithm for Index Reduction with Mixed Polynomial Matrices%
]{%
  Combinatorial Relaxation Algorithm for Index Reduction\\%
  with Mixed Polynomial Matrices%
}
\label{sec:algorithm}

This section presents our index reduction algorithm for a DAE
\begin{align} \label{eq:laplaced_mixed_dae}
  A(s) \tilde{x}(s) = \hat{f}(s)
\end{align}
with a nonsingular mixed polynomial matrix $A(s)$, which is the Laplace transform of the DAE~\eqref{def:linear_dae}.
From \cref{cor:MS_equivalent_validity_condition}, our goal is to find a unimodular matrix $U(s)$ such that $\bar{A}(s) = U(s) A(s)$ is upper-tight.
Then applying the MS-algorithm to the DAE $U(s) A(s) \tilde{x}(s) = U(s) \hat{f}(s)$, we obtain a resultant low-index DAE.

We cannot perform row operations on $A(s)$ involving rows containing independent parameters.
Our first step is to convert a given DAE~\eqref{eq:laplaced_mixed_dae} into another DAE whose coefficient matrix $A(s)$ is an LM-polynomial matrix expressed as $A(s) = \binom{Q(s)}{T(s)}$.
Then we can transform $A(s)$ to
\begin{align} \label{eq:update_A}
  \bar{A}(s)
  =
  \begin{pmatrix} U_Q(s) & O \\ O & I \end{pmatrix}
  \binom{Q(s)}{T(s)},
\end{align}
where $U_Q(s)$ is a unimodular matrix.
Note that we are allowed to perform row operations only on $Q(s)$ even for an LM-polynomial matrix $A(s) = \binom{Q(s)}{T(s)}$, and thus we cannot always reduce the index to one only by row operations on $Q(s)$.
We describe this conversion process from mixed polynomial matrices into LM-polynomial matrices in \cref{sec:mixed_to_LM}.

After the conversion, we find a unimodular matrix $U_Q(s)$ in~\eqref{eq:update_A} such that $A(s)$ is upper-tight based on the combinatorial relaxation approach.
The outline of our algorithm is as follows.

\paragraph{Algorithm for Tightness}
\begin{enumerate}[label={Phase \arabic*.}]
  \item
    Construct an optimal solution $\prn{p, q}$ of $\D{A}$ satisfying $0 \leq p_i \leq ln$ and $0 \leq q_j \leq ln$ for all $i \in R$ and $j \in C$, where $l$ is the maximum degree of an entry in $A(s)$.
  \item
    If the tight coefficient matrix $A^\#$ with respect to $\prn{p, q}$ is nonsingular, then return $A(s)$ and halt.
  \item
    Modify $A(s)$ into $\bar{A}(s)$ such that $\tdegdet{\bar{A}} \leq \tdegdet{A} - 1$ and $\degdet{A} = \degdet{\bar{A}}$.
    Update $\prn{p, q}$ to an optimal solution of $\D{\bar{A}}$, and go back to Phase~2.
\end{enumerate}

The bounds on $p_i$ and $q_j$ in Phase~1 are needed to bound the time complexity of our algorithm.
An algorithm to find such $(p, q)$ is given in~\cite{Iwata2018a} for $l = 1$, and we give an algorithm for general $l$ in \Cref{sec:construct_dual_opt}.
The condition in Phase~2, which is equivalent to the upper-tightness of $A(s)$ by \cref{lem:tightness1}, can be checked by solving an independent matching problem~\cite{Murota1987}.
The matrix modification and an update procedure of $\prn{p, q}$ in Phase 3 are explained in \cref{sec:matrix_modification,sec:dual_updates}, respectively.
In \cref{sec:complexity_analysis}, we analyze the time complexity of our algorithm.

\subsection{Reduction to LM-polynomial Matrices}
\label{sec:mixed_to_LM}

We first convert the DAE~\eqref{eq:laplaced_mixed_dae} with a mixed polynomial coefficient matrix $A(s) = Q(s) + T(s)$ into the following augmented DAE
\begin{align} \label{eq:LM_dae}
  \begin{pmatrix}
     I &  Q(s) \\
    -D & DT(s)
  \end{pmatrix}
  \begin{pmatrix}
    \tilde{y}(s) \\
    \tilde{z}(s)
  \end{pmatrix}
  =
  \begin{pmatrix}
    \hat{f}(s) \\
    0
  \end{pmatrix},
\end{align}
where $D$ is a diagonal matrix whose diagonal entries are independent parameters $\tau_1, \ldots, \tau_n$.
Note that the coefficient matrix of the augmented DAE~\eqref{eq:LM_dae} is an LM-polynomial matrix as the set of nonzero coefficients of entries in $-D$ and $DT(s)$ is algebraically independent over $\setK$.

\begin{proposition}
  Let $\binom{\tilde{y}(s)}{\tilde{z}(s)}$ be a solution of the DAE~\eqref{eq:LM_dae}.
  Then $\tilde{z}(s)$ is a solution of the DAE~\eqref{eq:laplaced_mixed_dae}.
\end{proposition}
\begin{proof}
  By left-multiplying both sides of~\eqref{eq:LM_dae} by a nonsingular constant matrix $\begin{pmatrix} I & O \\ I & D^{-1} \end{pmatrix}$, we obtain
  \begin{align}
    \begin{pmatrix}
      I & Q(s) \\
      O & A(s)
    \end{pmatrix}
    \begin{pmatrix}
      \tilde{y}(s) \\
      \tilde{z}(s)
    \end{pmatrix}
    =
    \begin{pmatrix}
      \hat{f}(s) \\
      \hat{f}(s)
    \end{pmatrix},
  \end{align}
  where $O$ is a zero matrix.
  Thus it holds $A(s) \tilde{z}(s) = \hat{f}(s)$, which implies that $\tilde{z}(s)$ is a solution of the DAE~\eqref{eq:laplaced_mixed_dae}.
\end{proof}

After the index reduction process, we need to fill independent parameters by real numbers to start a numerical method.
Indeed, we can substitute 1 for each diagonal entry $\tau_i$ of $D$, i.e., $D=I$.
To explain this fact, let 
\begin{align} \label{def:B}
  B(s) = \begin{pmatrix} Q_1(s) & Q_2(s) \\ -D & DT(s) \end{pmatrix}
\end{align}
be the coefficient matrix of a DAE that our algorithm returns for the augmented DAE~\eqref{eq:LM_dae}, where $Q_1(s)$ and $Q_2(s)$ are some polynomial matrices.
By substituting the identity matrix to $D$, we obtain
\begin{align} \label{def:tilde_B}
  \bar{B}(s) = \begin{pmatrix} Q_1(s) & Q_2(s) \\ -I & T(s) \end{pmatrix}.
\end{align}
Though $\bar{B}(s)$ is no longer an LM-polynomial matrix, the following lemma guarantees the upper-tightness of $\bar{B}(s)$.
\begin{lemma} \label{lem:B}
  Let $Q_1(s)$, $Q_2(s)$ and $T(s)$ be polynomial matrices and let $D$ be a nonsingular diagonal matrix.
  Then $B(s)$ in~\eqref{def:B} is upper-tight if and only if $\bar{B}(s)$ in~\eqref{def:tilde_B} is upper-tight.
\end{lemma}
\begin{proof}
  Using $P = \begin{pmatrix} I & O \\ O & D^{-1} \end{pmatrix}$, we have $\bar{B}(s) = P B(s)$.
  Since $P$ is a nonsingular constant matrix, $\degdet{B} = \degdet{\bar{B}}$ holds.
  In addition, since $P$ is nonsingular, diagonal and constant, the row transformation by $P$ does not change the bipartite graph $G(B)$ and its edge weight $c_e$ associated with $B(s)$.
  This fact implies that $\tdegdet{B} = \tdegdet{\bar{B}}$.
  Thus the upper-tightness of $B(s)$ and $\bar{B}(s)$ are equivalent.
\end{proof}
From this lemma, we can ``forget'' the existence of $D$ in the augmented DAE~\eqref{eq:LM_dae}.
That is, to reduce the index of the DAE~\eqref{eq:laplaced_mixed_dae}, it suffices to apply our algorithm to the DAE
\begin{align} \label{eq:eco_dae}
  \begin{pmatrix}
     I & Q(s) \\
    -I & T(s)
  \end{pmatrix}
  \begin{pmatrix}
    \tilde{y}(s) \\
    \tilde{z}(s)
  \end{pmatrix}
  =
  \begin{pmatrix}
    \hat{f}(s) \\
    0
  \end{pmatrix},
\end{align}
as if the set of nonzero coefficients of entries in $\begin{pmatrix} -I & T(s) \end{pmatrix}$ were independent.

\begin{example}
  Consider the index-2 DAE
  \begin{align} \label{eq:dae_example_2}
    \begin{pmatrix}
       1 & s + \alpha_1 \\
      -1 & -s + \alpha_2
    \end{pmatrix}
    \begin{pmatrix}
      \tilde{x}_1(s) \\
      \tilde{x}_2(s)
    \end{pmatrix}
    =
    \begin{pmatrix}
      \hat{f}_1(s) \\
      \hat{f}_2(s)
    \end{pmatrix},
  \end{align}
  where $\alpha_1$ and $\alpha_2$ are independent parameters.
  Following~\eqref{eq:eco_dae}, we convert this DAE into
  \begin{align} \label{eq:eae_example_2_LM}
    \prn{\begin{array}{cc|cc}
      1 &   &  1 & s \\
        & 1 & -1 & -s \\
      \hline
      -1 &    &   & \alpha_1 \\
         & -1 &   & \alpha_2
    \end{array}}
    \begin{pmatrix}
      \tilde{y}_1(s) \\
      \tilde{y}_2(s) \\
      \tilde{z}_1(s) \\
      \tilde{z}_2(s) \\
    \end{pmatrix}
    =
    \begin{pmatrix}
      \hat{f}_1(s) \\
      \hat{f}_2(s) \\
      0 \\
      0 \\
    \end{pmatrix}.
  \end{align}
  Then we can obtain a solution $\prn{\tilde{x}_1(s), \tilde{x}_2(s)}$ of~\eqref{eq:dae_example_2} by solving the augmented DAE~\eqref{eq:eae_example_2_LM}.
  While the index of~\eqref{eq:eae_example_2_LM} is also three, in general this conversion does not preserve the index of DAEs.
\end{example}

\subsection{Construction of Dual Optimal Solution}
\label{sec:construct_dual_opt}

Let $A(s)$ be an $n \times n$ nonsingular LM-polynomial matrix with $R = \Row(A)$ and $C = \Col(A)$, and let $l$ be the maximum degree of an entry in $A(s)$.
An optimal solution $\prn{p, q}$ of $\D{A}$ satisfying $0 \leq p_i \leq ln$ and $0 \leq q_j \leq ln$ for all $i \in R$ and $j \in C$ is constructed as follows.

First, we obtain a maximum-weight perfect matching $M \subseteq E(A)$ in $G(A)$ by the Hungarian method~\cite{Kuhn1955}.
Next, construct a residual graph $G_M = (W, E_M)$ with $W = R \cup C \cup \set{r}$ and $E_M = \reor{E} \cup M \cup Z$, where $r$ is a new vertex, $\reor{E} = \set{\prn{j, i}}[\prn{i, j} \in E(A)]$, and $Z = \set{\prn{r, i}}[i \in R]$.
The arc length $\funcdoms{\gamma}{E_M}{\setZ}$ of $G_M$ is defined by
\begin{align}
  \gamma(i, j) = \begin{cases}
    -c_{j, i} & \prn{\prn{i, j} \in \reor{E}}, \\
     c_{i, j} & \prn{\prn{i, j} \in M}, \\
     0        & \prn{\prn{i, j} \in Z} \\
  \end{cases}
\end{align}
for each $\prn{i, j} \in E_M$.

\begin{lemma} \label{lem:optimality_mincost_flow}
  For the residual graph $G_M$ defined above, the following hold.
  \begin{enumerate}[label={(\arabic*)}]
    \item All vertices are reachable from $r$.
    \item There is no negative-weight directed cycle with respect to $\gamma$.
  \end{enumerate}
\end{lemma}
\begin{proof}
  (1)
  Every vertex $i \in R$ is reachable from $r$ through an edge $\prn{r, i} \in Z$.
  In addition, since $G(A)$ has a perfect matching $M$, every vertex $j \in C$ is also reachable from $r$ via $i \in R$ through edges $\prn{r, i} \in Z$ and $\prn{i, j} \in M \subseteq E(A)$.
  
  (2)
  This immediately follows from an optimality criterion~\cite[Theorem~9.6]{Korte2008} of the minimum cost flow problem.
\end{proof}

For $i, j \in W$ such that $i$ is reachable to $j$, let $d(i, j)$ denote the length of a shortest path from $i$ to $j$ with respect to the arc length $\gamma$ in $G_M$.
Lemma~4.4 guarantees that $d(r, v)$ is defined for all $v \in W$.
Using $d$, we define
\begin{align} 
  p_i &\defeq d(r, i) - \min_{i^* \in R} d(r, i^*), \label{def:initial_p} \\
  q_j &\defeq d(r, j) - \min_{i^* \in R} d(r, i^*)  \label{def:initial_q}
\end{align}
for each $i \in R$ and $j \in C$.

The next lemma is easily shown in almost the same way as the case for $l = 1$ in~\cite[Lemma~2.2]{Iwata2018a}.

\begin{lemma} \label{lem:bound_pq}
  Let $\prn{p, q}$ be defined in~\eqref{def:initial_p} and~\eqref{def:initial_q}.
  Then $\prn{p, q}$ is an optimal solution of $\D{A}$ satisfying $0 \leq p_i \leq ln$ for each $i \in R$ and $0 \leq q_j \leq ln$ for each $j \in C$.
\end{lemma}
\begin{proof}
  First, we prove that $\prn{p, q}$ is a feasible solution of $\D{A}$.
  By the definition of $\prn{p, q}$, every $p_i \, \prn{i \in R}$ and $q_j \, \prn{j \in C}$ are clearly integer.
  For each $\prn{i, j} \in E(A)$, it holds $d(r, i) \leq d(r, j) - c_{i, j}$.
  Thus
  \begin{align}
    q_j - p_i = d(r, j) - d(r, i) \geq c_{i, j}
  \end{align}
  and this implies that $\prn{p, q}$ is a feasible solution of $\D{A}$.
  
  We second show the optimality of $\prn{p, q}$.
  For each $\prn{i, j} \in M$, since $\prn{i, j} \in E_M$ and $\prn{j, i} \in E_M$, we obtain
  \begin{align}
    q_j - p_i = d(r, j) - d(r, i) = c_{i, j}.
  \end{align}
  Thus it holds that
  \begin{align}
    \sum_{j \in C} q_j - \sum_{i \in R} p_i
    = \sum_{j \in C} d(r, j) - \sum_{i \in R} d(r, i)
    = \sum_{\prn{i, j} \in M} \prn{d(r, j) - d(r, i)}
    = \sum_{\prn{i, j} \in M} c_{i,j}
  \end{align}
  which implies that $\prn{p, q}$ is optimal to $\D{A}$.
  
  Finally, we give the lower and upper bounds on $p_i$ and $q_j$.
  The non-negativity of $p_i$ clearly follows from the definition of $p_i$.
  In addition, since $G(A)$ has a perfect matching, each $j \in C$ is incident to at least one vertex $i \in R$ on $G(A)$.
  Thus we obtain $q_j \geq p_i + c_{i,j} \geq 0$ by $p_i, c_{i, j} \geq 0$.
  Let $i^* \in R$ denote a vertex such that $d(r, i^*) \leq d(r, i)$ for all $i \in R$.
  Fix $j \in C$.
  Let $P_j \subseteq E_M$ and $P_{i^*} \subseteq E_M$ be shortest paths from $r$ to $j$ and $i^*$, respectively.
  Let $v \in W$ be the last common vertex in $P_j$ and $P_{i^*}$.
  Then it holds $q_j = d(r, j) - d(r, i^*) = d(v, j) - d(v, i^*)$.
  Let $Q_j \subseteq P_j$ and $Q_{i^*} \subseteq P_{i^*}$ denote subpaths from $v$ to $j$ and $i^*$, respectively.
  Note that $d(v, j)$ is at most $l$ times the number of edges in $E(A)$ on $Q_j$, whereas $-d(v, i^*)$ is at most $l$ times the number of edges in $\reor{M}$ on $Q_{i^*}$.
  The sum of these upper bounds is at most $ln$ since $Q_{i^*}$ and $Q_j$ do not share the same vertex besides $v$.
  Thus $q_j \leq ln$ holds for each $j \in C$.
  In addition, for each $i \in R$, we have $p_i \leq q_j - c_{i,j} \leq q_j \leq ln$, where $j \in C$ is incident to $i$ in $M$.
\end{proof}

\begin{example}
  Consider the coefficient matrix
  \begin{align} \label{eq:dae_example_coefficient}
    A(s) = 
    \begin{pmatrix}
       1 &    &  1 & s \\
         &  1 & -1 & -s \\
      -1 &    &    & \alpha_1 \\
         & -1 &    & \alpha_2
    \end{pmatrix}
  \end{align}
  in the DAE~\eqref{eq:eae_example_2_LM}.
  An optimal solution of the assignment problem $\P{A}$ is given by
  \begin{align}
    M = \set{(1, 3), (2, 4), (3, 1), (4, 2)}
  \end{align}
  with optimal value $\tdegdet{A} = 1$.
  \Cref{fig:residual_graph_example} shows the residual graph $G_M$ for $M$.
  According to~\eqref{def:initial_p} and~\eqref{def:initial_q}, a dual optimal solution $(p, q)$ is calculated as $p = (0, 0, 0, 0)$ and $q = (0, 0, 0, 1)$.
\end{example}

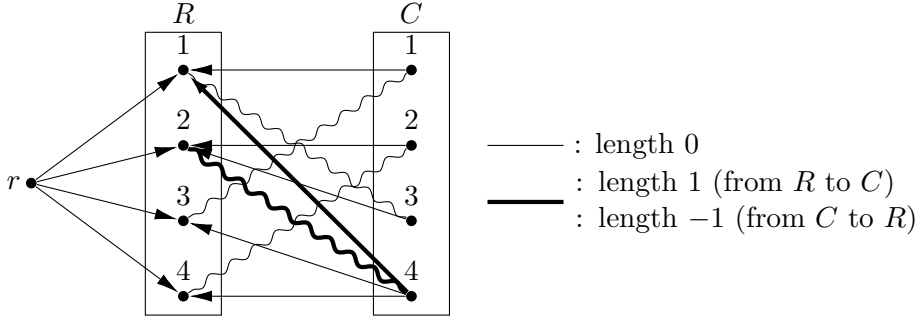
\begin{figure}[tbp]
  \centering
  \begin{tikzpicture}[x=1cm, y=0.5cm]
    \tikzset{node/.style={fill, circle, inner sep=0pt, minimum size=4pt}};
    \node (r)  [node] at (0, 0) {} node [left] {$r$};
    \node (p1) [node] at (2,  3) {}; \node [above=0mm of p1] {1};
    \node (p2) [node] at (2,  1) {}; \node [above=0mm of p2] {2};
    \node (p3) [node] at (2, -1) {}; \node [above=0mm of p3] {3};
    \node (p4) [node] at (2, -3) {}; \node [above=0mm of p4] {4};
    \node (q1) [node] at (5,  3) {}; \node [above=0mm of q1] {1};
    \node (q2) [node] at (5,  1) {}; \node [above=0mm of q2] {2};
    \node (q3) [node] at (5, -1) {}; \node [above=0mm of q3] {3};
    \node (q4) [node] at (5, -3) {}; \node [above=0mm of q4] {4};
    \tikzset{matching/.style={decorate, decoration={snake,segment length=3.5mm, amplitude=0.5mm}}}
    \tikzset{unmatching/.style={-{Latex[length=3mm, width=1.5mm]}}}
    \draw[unmatching] (r) -- (p1);
    \draw[unmatching] (r) -- (p2);
    \draw[unmatching] (r) -- (p3);
    \draw[unmatching] (r) -- (p4);
    \draw[unmatching] (q1) -- (p1);
    \draw[matching] (q3) -- (p1);
    \draw[unmatching, ultra thick] (q4) -- (p1);
    \draw[unmatching] (q2) -- (p2);
    \draw[unmatching] (q3) -- (p2);
    \draw[matching, ultra thick] (p2) -- (q4);
    \draw[matching] (q1) -- (p3);
    \draw[unmatching] (q4) -- (p3);
    \draw[matching] (q2) -- (p4);
    \draw[unmatching] (q4) -- (p4);
    \draw (1.5, 4) rectangle (2.5, -3.5);
    \node at (2, 4.5) {$R$};
    \draw (4.5, 4) rectangle (5.5, -3.5);
    \node at (5, 4.5) {$C$};
    \draw (6, 1) -- (7, 1) node [right] {: length 0};
    \draw[ultra thick] (6, -0.5) -- (7, -0.5) node [right, text width=5cm] {: length $1$ (from $R$ to $C$)\\: length $-1$ (from $C$ to $R$)};
  \end{tikzpicture}
  \caption{%
    The residual graph $G_M$ of~\eqref{eq:dae_example_coefficient} with $M = \set{(1, 3), (2, 4), (3, 1), (4, 2)}$.
    Edges in $M$, which are shown by wavy curves, are bidirectional and have lengths whose signs reverse according to the direction.
  }
  \label{fig:residual_graph_example}
\end{figure}

\subsection{Matrix Modification}
\label{sec:matrix_modification}

Let $A(s) = \binom{Q(s)}{T(s)}$ be an $n \times n$ nonsingular LM-polynomial matrix that is not upper-tight.
Let $A^\# = \binom{Q^\#}{T^\#}$ be the tight coefficient matrix with respect to an optimal solution $\prn{p, q}$ of $\D{A}$.
Without loss of generality, we assume that $\Row(Q) = R_Q = \set{1, \ldots, m_Q}$ and $p_1 \leq \cdots \leq p_{m_Q}$, where $m_Q = \card{R_Q}$.

Recall the rank identity~\eqref{eq:rank_identity_min}.
Let $J^* \subseteq C$ be a column subset that minimizes the right-hand side of the identity for $A^\#$, i.e., it holds
\begin{align} \label{eq:Jstar}
  \rank A^\# = \rank Q^\# [R_Q, J^*] + \trank T^\# [R_T, J^*] + \card{C \setminus J^*}.
\end{align}
Such $J^*$ is called a $\emph{minimizer}$ of~\eqref{eq:rank_identity_min}.
By a row transformation of $Q^\#$, we obtain a matrix $\bar{Q}^\# = UQ^\#$ such that
\begin{align} \label{eq:rank_Qsharp}
  \rank \bar{Q}^\#[R_Q, J^*] = \trank \bar{Q}^\#[R_Q, J^*].
\end{align}
In particular, this transformation can be accomplished only by operations of adding a scalar multiple of a row $i \in R_Q$ to another row $j \in R_Q$ with $p_i > p_j$.
Then the matrix $U$ is upper-triangular due to the order of rows in $R_Q$.
This is the \emph{forward elimination} on $\bar{Q}^\#[R_Q, J^*]$ with the order of the rows reversed.
Consider
\begin{align} \label{def:U_tilde}
  U_Q(s) =
  \diag(s^{-p_1}, \ldots, s^{-p_{m_Q}})
  U
  \diag(s^{p_1}, \ldots, s^{p_{m_Q}}),
\end{align}
where $\diag(a_1, \ldots, a_n)$ denotes a diagonal matrix with diagonal entries $a_1, \ldots, a_n$.
Note that each entry in $U_Q(s)$ is a polynomial because $U$ is upper-triangular.
In addition, since $\det U_Q(s) = \det U$ is a nonzero constant, $U_Q(s)$ is unimodular.

We define $D_p(s) = \diag(s^{p_1}, \ldots, s^{p_n})$ and $D_q(s)=\diag(s^{q_1}, \ldots, s^{q_n})$. 
Using $U_Q(s)$, we update $A(s)$ to $\bar{A}(s)$ as in~\eqref{eq:update_A}: 
\begin{align} \label{eq:update_A_2}
  \bar{A}(s) =
  \begin{pmatrix}
    U_Q(s) & O \\
    O      & I
  \end{pmatrix}
  A(s)
  =
  D_p^{-1}(s)
  \begin{pmatrix}
    U & O \\
    O & I
  \end{pmatrix}
  D_p(s)
  A(s).
\end{align}

To show that $\prn{p, q}$ is not an optimal solution of $\D{\bar{A}}$, we use the following lemma, which is given by Murota~\cite{Murota1990} as a combinatorial counterpart to \cref{lem:tightness1}.

\begin{lemma}[{\cite[Proposition 6.2]{Murota1990}}] \label{lem:tightness2}
  Let $A(s)$ be an $n \times n$ nonsingular polynomial matrix and let $A^\#$ be the tight coefficient matrix of $A(s)$ with respect to a feasible solution $\prn{p, q}$ of $\D{A}$.
  Then $\prn{p, q}$ is optimal if and only if $\trank A^\# = n$.
\end{lemma}

\begin{lemma} \label{lem:tilde_A_optimality}
  Let $A(s) = \binom{Q(s)}{T(s)}$ be an $n \times n$ nonsingular LM-polynomial matrix that is not upper-tight, and $A^\# = \binom{Q^\#}{T^\#}$ the tight coefficient matrix with respect to an optimal solution $\prn{p, q}$ of $\D{A}$.
  Then for the LM-polynomial matrix $\bar{A}(s)$ defined in~\eqref{eq:update_A_2}, the value $\prn{p, q}$ is feasible on $\D{\bar{A}}$ but not optimal.
\end{lemma}
\begin{proof}
  Consider a rational function matrix
  \begin{align} \label{def:F}
    H(s) = D_p(s) \bar{A}(s) D_q^{-1}(s).
  \end{align}
  For each $i \in R$ and $j \in C$, it holds that $\deg H_{i,j}(s) = \bar{c}_{i, j} + p_i - q_j$, where $\bar{c}_{i,j} = \deg \bar{A}_{i,j}(s)$.
  By substituting~\eqref{eq:update_A_2} into~\eqref{def:F}, we obtain
  \begin{align}
    H(s) &= \begin{pmatrix}
      U & O \\
      O & I
    \end{pmatrix} D_p(s) A(s) D_q^{-1}(s)
    = \begin{pmatrix}
      U & O \\
      O & I
    \end{pmatrix} \prn{A^\# + A^\infty(s)},
  \end{align}
  where $A^\infty(s)$ is a matrix whose entries are polynomials in $s^{-1}$ without
constant terms.
  Hence for each $i \in R$ and $j \in C$, it holds $\deg H_{i,j}(s) \leq 0$, which implies $\bar{c}_{i, j} \leq q_j - p_i$.
  Therefore $\prn{p, q}$ is feasible on $\D{\bar{A}}$.
  
  Next, we show that $\prn{p, q}$ is not optimal on $\D{\bar{A}}$.
  From~\eqref{eq:update_A_2}, the tight coefficient matrix $\bar{A}^\#$ of $\bar{A}(s)$ with respect to $\prn{p, q}$ is
  \begin{align} \label{def:tilde_A_sharp}
    \bar{A}^\#
    =
    \begin{pmatrix} U & O \\ O & I \end{pmatrix} A^\#
    =
    \begin{pmatrix} \bar{Q}^\# \\ T^\# \end{pmatrix},
  \end{align}
  where $\bar{Q}^\# = UQ^\#$.
  From \cref{prop:trank_identity_min} and~\eqref{eq:rank_Qsharp}, it holds
  \begin{align}
    \trank \bar{A}^\#
    &= \min \set{\trank \bar{Q}^\#[R_Q, J] + \trank T^\#[R_T, J] + \card{C \setminus J}}[J \subseteq C] \\
    &\leq \trank \bar{Q}^\#[R_Q, J^*] + \trank T^\#[R_T, J^*] + \card{C \setminus J^*} \\
    &= \rank \bar{Q}^\#[R_Q, J^*] + \trank T^\#[R_T, J^*] + \card{C \setminus J^*}.
  \end{align}
  Now since $Q^\#[R_Q, J^*]$ and $\bar{Q}^\#[R_Q, J^*] = UQ^\#[R_Q, J^*]$ have the same rank, we obtain
  \begin{align}
    \trank \bar{A}^\#
    &\leq \rank Q^\#[R_Q, J^*] + \trank T^\#[R_T, J^*] + \card{C \setminus J^*} = \rank A^\#,
  \end{align}
  where the last equality comes from~\eqref{eq:Jstar}.
  In addition, since $\rank \bar{A}^\# = \rank A^\#$ from~\eqref{def:tilde_A_sharp}, we have $\trank \bar{A}^\# \leq \rank \bar{A}^\#$, which implies $\trank \bar{A}^\# = \rank \bar{A}^\# = \rank A^\#$.
  Furthermore, since $A(s)$ is not upper-tight, we have $\rank A^\# < n$ by \cref{lem:tightness1}. 
  Thus, $\trank A^\# = \rank A^\# < n$ holds.
  It then follows from \cref{lem:tightness2} that $\prn{p, q}$ is not optimal on $\D{\bar{A}}$.
\end{proof}

From \cref{lem:tilde_A_optimality} and the unimodularity of $U_Q(s)$, we obtain the following.

\begin{corollary}
  Let $A(s) = \binom{Q(s)}{T(s)}$ be an $n \times n$ nonsingular LM-polynomial matrix that is not upper-tight, and $\bar{A}(s)$ the LM-polynomial matrix defined in~\eqref{eq:update_A_2}.
  Then $\tdegdet{\bar{A}} \leq \tdegdet{A} - 1$ and $\degdet{A} = \degdet{\bar{A}}$ hold.
\end{corollary}
\begin{example}
  Consider the LM-polynomial matrix~\eqref{eq:dae_example_coefficient} again.
  The tight coefficient matrix $A^\#$ with respect to $p = (0, 0, 0, 0)$ and $q = (0, 0, 0, 1)$ is
  \begin{align}
    A^\#
    = \begin{pmatrix}
      Q^\# \\
      T^\#
    \end{pmatrix}
    =
    \begin{pmatrix}
      1 &   & 1 & 1 \\
        & 1 & -1 & -1 \\
      -1 &    &   &   \\
         & -1 &   &  
    \end{pmatrix},
  \end{align}
  where the row sets $R_Q$ of $Q^\#$ and $R_T$ of $T^\#$ correspond to the first and last two rows in $A^\#$, respectively.
  A minimizer $J^* \subseteq C$ is the set of the right two columns as follows:
  \begin{align}
    A^\# = \begin{blockarray}{ccccc}
      \overmat{2}{1 & 1}{C \setminus J^*} & \overmat{2}{-1 & -1}{J^*}
      \\
      \begin{block}{(cccc)c}
         1 &    &  1 &  1 & \rightmat{2}{R_Q} \\
           &  1 & -1 & -1 & \\
        -1 &    &    &    & \rightmat{2}{R_T} \\
           & -1 &    &    & \\
      \end{block}
    \end{blockarray}
    \quad \quad.
  \end{align}
  Then the rank of $A^\#$ is calculated by~\eqref{eq:Jstar} as $Q^\#[R_Q, J^*] + T^\#[R_T, J^*] + \card{C \setminus J^*} = 1+0+2 = 3$.
  Since $A^\#$ is not upper-tight, we need to modify $A(s)$.
  By performing Gaussian elimination on $Q^\#[R_Q, J^*] = \begin{pmatrix} 1 & 1 \\ -1 & -1 \end{pmatrix}$, we obtain
  \begin{align}
    \bar{Q}^\#[R_Q, J^*] = U Q^\#[R_Q, J^*] = \begin{pmatrix}  &  \\ -1 & -1 \end{pmatrix},
  \end{align}
  where $U = \begin{pmatrix} 1 & 1 \\   & 1 \end{pmatrix}$.
  The unimodular matrix $U_Q(s)$ defined by~\eqref{def:U_tilde} coincides with $U$ since all $p_i$ are zero.
  According to~\eqref{eq:update_A_2}, we update $A(s)$ into
  \begin{align} \label{eq:example_modified_As}
    \bar{A}(s) = 
      \begin{pmatrix}
        U_Q(s) & O \\
        O & I
      \end{pmatrix}
      A(s)
      =
      \begin{pmatrix}
        1 & 1 &   &  \\
          & 1 &   &  \\
          &   & 1 &  \\
          &   &   & 1   
      \end{pmatrix}
      \begin{pmatrix}
       1 &    &  1 & s \\
         &  1 & -1 & -s \\
      -1 &    &    & \alpha_1 \\
         & -1 &    & \alpha_2
    \end{pmatrix}
    =
    \begin{pmatrix}
       1 &  1 &    &   \\
         &  1 & -1 & -s \\
      -1 &    &    & \alpha_1 \\
         & -1 &    & \alpha_2
    \end{pmatrix}.
  \end{align}
\end{example}

\subsection{Dual Updates}
\label{sec:dual_updates}

Let $\prn{p, q}$ be a feasible solution of $\D{\bar{A}}$.
We obtain an optimal solution of $\D{\bar{A}}$ by iterating the following procedure.

Let $\bar{A}^\#$ be the tight coefficient matrix of $\bar{A}(s)$ with respect to $\prn{p, q}$.
First we check if $\trank \bar{A}^\# = n$.
If it is, $\prn{p, q}$ is an optimal solution of $\D{\bar{A}}$ from \cref{lem:tightness2} and we are done.
Otherwise, we construct a feasible solution $(p', q')$ of $\D{\bar{A}}$ such that the difference
\begin{align} \label{def:Delta}
  \Delta \defeq \sum_{j \in C} (q'_j - q_j) - \sum_{i \in R} (p'_i - p_i)
\end{align}
of the objective values is negative.
Let $G^\# = \prn{R \cup C, E^\#}$ be a bipartite graph defined by
\begin{align}
  E^\#
  \defeq \set[\big]{\prn{i, j} \in R \times C}[\bar{A}^\#_{i,j} \neq 0]
  =      \set{\prn{i, j} \in E(\bar{A})}[q_j - p_i = \bar{c}_{i,j}].
\end{align}
Since $\prn{p, q}$ is not optimal, there is no perfect matching of $G^\#$.
Thus $G^\#$ has a vertex cover $S \subseteq R \cup C$ with $\card{S} < n$ by the K\"{o}nig--Egerv\'{a}ry theorem.
Using this $S$, we define $\prn{p', q'}$ as follows:
\begin{align} \label{eq:update_pq}
  p'_i = \begin{cases}
    p_i     & \prn{i \in R \cap S} \\
    p_i + 1 & \prn{i \in R \setminus S}
  \end{cases}
  , \quad
  q'_j = \begin{cases}
    q_j + 1 & \prn{j \in C \cap S} \\
    q_j     & \prn{j \in C \setminus S}
  \end{cases}
\end{align}
for $i \in R$ and $j \in C$.

The following lemma is a restatement of a well-known fact~\cite{Kuhn1955}.
We give a proof for completeness.

\begin{lemma} \label{lem:dual_update}
  Let $\prn{p, q}$ be a feasible but not optimal solution of $\D{\bar{A}}$ and $\prn{p', q'}$ defined in~\eqref{eq:update_pq}.
  Then the difference $\Delta$ of the objective values in~\eqref{def:Delta} is negative, and $\prn{p', q'}$ is a feasible solution of $\D{\bar{A}}$.
\end{lemma}
\begin{proof}
  The difference of the objective values is $\Delta = \card{C \cap S} - \card{R \setminus S} = \card{S} - \card{R} < 0$.
  Next, we show the feasibility of $\prn{p', q'}$.
  For every $\prn{i, j} \in E(\bar{A})$, it holds $q_j - p_i \geq \bar{c}_{i,j}$ since $\prn{p, q}$ is feasible.
  If $i \in S$ or $j \in S$, it holds $(q'_j - q_j) - (p'_i - p_i) \geq 0$, which imply $q'_j - p'_i \geq q_j - p_i \geq \bar{c}_{i,j}$.
  If $i \notin W$ and $j \notin W$, then $\prn{i, j}$ is not an edge of $G^\#$ since $\prn{i, j}$ is not covered by $W$.
  Hence it holds $q_j - p_i > \bar{c}_{i,j}$, and thus $q'_j - p'_i = q_j - p_i - 1 \geq \bar{c}_{i,j}$.
\end{proof}

We update $\prn{p, q}$ to $\prn{p', q'}$, and go back to the optimality checking.
From \cref{lem:dual_update}, it is guaranteed that $\prn{p, q}$ becomes an optimal solution of $\D{\bar{A}}$ by iterating the update process above.

\begin{example}
  Consider the modified LM-polynomial matrix~\eqref{eq:example_modified_As}.
  The tight coefficient matrix $\bar{A}^\#$ of $\bar{A}(s)$ with respect to $p = (0, 0, 0, 0)$ and $q = (0, 0, 0, 1)$ is
  \begin{align}
    A^\#
    =
    \begin{pmatrix}
       1 &  1 &    &    \\
         &  1 & -1 & -1 \\
      -1 &    &    &    \\
         & -1 &    & 
    \end{pmatrix}.
  \end{align}
  Let $S$ be the set of the first and second columns and the second row of $A^\#$.
  Then $S$ is a vertex cover of $G^\#$ with $\card{S} < 4$.
  Following~\eqref{eq:update_pq}, we update $(p, q)$ to $p' = (1, 0, 1, 1)$ and $q' = (1, 1, 0, 1)$.
  We then go back to Phase~2 for $\bar{A}(s)$.
  It is indeed confirmed in the next iteration that $\bar{A}(s)$ is upper-tight, and thus the iteration ends at this point.
  We can obtain a low-index DAE by applying the MS-algorithm.
\end{example}

\subsection{Complexity Analysis}
\label{sec:complexity_analysis}

This section is devoted to complexity analysis.
The dominating part in our algorithm is the matrix multiplications in~\eqref{eq:update_A_2}.

Let $A(s)$ be an $n \times n$ nonsingular LM-polynomial matrix and let $A^\#$ be the tight coefficient matrix with respect to an optimal solution $\prn{p, q}$ of $\D{A}$.
From the definition of $A^\#$, we can express $A(s)$ as
\begin{align} \label{eq:As_neg}
  A(s) =
    D_p^{-1}(s)
    \prn{A^\# + \sum_{k=1}^K s^{-k} V_k}
    D_q(s)
\end{align}
for some $K$ matrices $V_1, V_2, \ldots, V_K$ with $V_K \neq O$.
By~\eqref{eq:update_A_2} and~\eqref{eq:As_neg}, we have
\begin{align}
  \bar{A}(s) =
    D^{-1}_p(s)
    \begin{pmatrix} U & O \\ O & I \end{pmatrix}
    \prn{A^\# + \sum_{k=1}^K s^{-k} V_k}
    D_q(s).
\end{align}
Therefore, we can compute $\bar{A}(s)$ by performing $K+1$ constant matrix multiplications.

By $V_K \neq O$, there exist $i \in R$ and $j \in C$ such that the $\prn{i, j}$ entry in $V_K$ is nonzero.
Then the degree of the corresponding term in $A_{i,j}(s)$ is equal to $q_j - p_i - K$.
Since $A_{i,j}(s)$ is a polynomial, we have $q_j - p_i - K \geq 0$, which implies $K \leq q_j - p_i \leq q_j$.
The following lemma bounds $p_i$ and $q_j$ at any iteration of our algorithm.

\begin{lemma} \label{lem:bound_l}
  During the algorithm, the values $p_i$ and $q_j$ are at most $2 ln$ for $i \in R$ and $j \in C$, where $l$ is the maximum degree of an entry in $A(s)$.
\end{lemma}
\begin{proof}
  From \cref{lem:bound_pq}, the initial values of $p_i$ and $q_j$ are bounded by $ln$.
  In every update of $\prn{p, q}$, the values $p_i$ and $q_j$ increase by at most one from the update rule~\eqref{eq:update_pq}.
  In addition, $\prn{p, q}$ is updated at most $\tdegdet{A} - \degdet{A} \leq ln$ times because the objective value $\sum_{j \in C} q_j - \sum_{i \in R} p_i$ of the dual problem decreases by at least one in every update.
  Therefore, at any iteration of the algorithm, it holds $p_i, q_j \leq ln + \tdegdet {A} \leq ln + ln = 2ln$.
\end{proof}

The time complexity of our algorithm is as follows.
\begin{theorem} \label{thm:complexity_tightness}
  Let $A(s)$ be an $n \times n$ nonsingular LM-polynomial matrix and let $l$ be the maximum degree of an entry in $A(s)$.
  Then Algorithm for Tightness runs in $\Order\prn{l^2 n^{\omega + 2}}$ time, where $2 < \omega \leq 3$ is the matrix multiplication exponent.
\end{theorem}

\begin{proof}
  Phase~1 can be done in $\Order\prn{n^3}$ time by the Hungarian method~\cite{Kuhn1955} and shortest path algorithms such as the Bellman--Ford algorithm.
  Consider the time complexity in every iteration of Phases~2 and~3.
  In Phase~2, the nonsingularity of the tight coefficient matrix $A^\#$ can be checked via the rank identity~\eqref{eq:rank_identity_min}.
  Thus an efficient way is to obtain a minimizer $J^*$ of~\eqref{eq:rank_identity_min} before Phase 2, and then check the nonsingularity of $A^\#$ by~\eqref{eq:rank_identity_min}.
  The minimizer $J^*$ can be found from a residual graph constructed by an augmenting path type algorithm~\cite{Murota1987}, which runs in $\Order\prn{n^3 \log n}$ time~\cite{Cunningham1986}.
  The computation of $\bar{A}(s)$ in Phase~3 can be done in $\Order\prn{N n^\omega} = \Order\prn[\big]{\max_{j \in C} q_j n^\omega} = \Order\prn{l n^{\omega + 1}}$ time from \cref{lem:bound_l}, where $\prn{p, q}$ is a dual optimal solution of $\D{A}$ and $N$ is in~\eqref{eq:As_neg}.
  In addition, since the number of iterations of Phases~2 and~3 is at most $\tdegdet{A} - \degdet{A} \leq ln$, the running time in Phases~2 and~3 is $\Order\prn{l^2 n^{\omega+2}}$.
  Finally, the updates of $\prn{p, q}$ run in $\Order\prn{ln^4}$ time: $\prn{p, q}$ is updated at most $\tdegdet{A} \leq ln$ times, and in every update, we can find a vertex cover in $\Order\prn{n^3}$ time by Ford--Fulkerson's algorithm.
  Thus the total running time is $\Order\prn{l^2 n^{\omega+2}}$.
\end{proof}

\begin{theorem} \label{thm:complexity_all}
  For a DAE~\eqref{eq:laplaced_mixed_dae} with $n \times n$ nonsingular mixed polynomial matrix $A(s)$, our algorithm returns an equivalent DAE of index zero or one in $\Order\prn{l^2 n^{\omega+2}}$ time, where $2 < \omega \leq 3$ is the matrix multiplication exponent and $l$ is the maximum degree of entries in $A(s)$.
\end{theorem}

\begin{proof}
  First we convert the DAE into an equivalent DAE with LM-polynomial matrix $\LM{A}(s)$ of size $2n \times 2n$.
  Note that the maximum degree of an entry in $\LM{A}(s)$ is equal to $l$ by~\eqref{eq:LM_dae}.
  Hence it follows from \cref{thm:complexity_tightness} that Algorithm for Tightness for $\LM{A}(s)$ runs in $\Order\prn{l^2 n^{\omega+2}}$ time.
  The resulting DAE has a coefficient matrix such that the maximum degree of an entry is at most $4ln$, because it holds that
  \begin{align}
    \deg \LM{A}_{i,j}(s) \leq q_j - p_i \leq q_j \leq 4ln
  \end{align}
  with a feasible solution $\prn{p, q}$ of $\D{\LM{A}}$, where the last inequality is due to \cref{lem:bound_l}.
  
  Next we analyze the complexity of the MS-algorithm described in \cref{sec:ms_alg}.
  In Step~1, we can reuse a dual optimal solution $\prn{p, q}$ obtained at the termination of Algorithm for Tightness, or compute a new $\prn{\tilde{p}, \tilde{q}}$ such that $\tilde{p}_i \leq p_i$ for $i \in R$ to decrease the number of dummy variables, in $\Order\prn{n^3}$ time.
  The nonsingularity of the corresponding tight coefficient matrix can be verified by solving an independent matching problem in $\Order\prn{n^3 \log n}$ time~\cite{Cunningham1986, Murota1987}.
  Step~2 runs in $\Order\prn{n^4 \log n}$ time since we solve independent matching problems at most $2n$ times.
  We now consider the resultant DAE returned in Step~4.
  The number of original (non-dummy) variables is $2n$, and from \cref{lem:bound_l}, the orders of their derivatives are at most
  \begin{align}
    4ln + \max_{i \in R} p_i \leq 4ln + 4ln = \Order{ln}.
  \end{align}
  In contrast, the number of dummy variables is $\sum_{i \in R} p_i = \Order\prn{ln^2}$, and there is no derivative of dummy variables in the resultant DAE.
  Therefore, the number of terms in the resultant DAE is $\Order\prn{l n^2}$, and thus Step~4 runs in $\Order\prn{ln^2}$ time.
  Hence the MS-algorithm costs $\Order\prn{n^4 \log n + l n^2}$ time.
  
  Since the bottleneck in the entire algorithm is Algorithm for Tightness, the total running time of our algorithm is $\Order\prn{l^2 n^{\omega+2}}$.
\end{proof}

\section{Exploiting Dimensional Consistency}
\label{sec:exploiting_dimensional_consistency}

\subsection{Dimensional Consistency}

The principle of dimensional homogeneity claims that any equation describing a physical phenomenon must be consistent with respect to physical dimensions.
To reflect the dimensional consistency in conservation laws of dynamical systems, Murota~\cite{Murota1985b} introduced a class of mixed polynomial matrices $A(s) = Q(s) + T(s)$ that satisfy the following condition:
\begin{enumerate}[label={(MP-DC)}]
  \item[(MP-DC)] $Q(s)$ is written as
    \begin{align} \label{eq:dimensional_consistency_1}
      Q(s) =
      \diag(s^{-\lambda_1}, \ldots, s^{-\lambda_{m}})
      Q(1)
      \diag(s^{\mu_1}, \ldots, s^{\mu_n})
    \end{align}
    for some integers $\lambda_1, \ldots, \lambda_{m}$ and $\mu_1, \ldots, \mu_n$.
\end{enumerate}
A mixed polynomial matrix satisfying (MP-DC) is said to be \emph{dimensionally consistent}.
We abbreviate a dimensionally consistent mixed polynomial matrix and a dimensionally consistent LM-polynomial matrix to a \emph{DCM-polynomial matrix} and a \emph{DCLM-polynomial matrix}, respectively.

\begin{example}
  Consider the DAE~\eqref{eq:example2_eq} representing the electrical circuit shown in \cref{fig:RLC}.
  Since $Q(s)$ of the coefficient matrix $A(s)$ of~\eqref{eq:example2_eq} is constant, $A(s)$ is a DCM-polynomial matrix with all $\lambda_i$ and $\mu_j$ being zero.
  Note that $(\lambda, \mu)$ is not uniquely determined; the values
  \begin{align} \label{eq:lambda_and_mu_of_RLC}
    \lambda = (0,0,-3,-3,-3,-3,-3,-3,0,-3), \quad \mu = (0,0,0,0,0,-3,-3,-3,-3,-3)
  \end{align}
  also satisfy~\eqref{eq:dimensional_consistency_1}.
\end{example}

The condition (MP-DC) can be ``derived'' from physical observations as follows.
Suppose that a DAE $A(s)\tilde{x}(s) = \hat{f}(s)$ arises from a dynamical system and the $i$th equation and the $j$th variable have physical dimensions $X_i$ and $Y_j$, respectively.
For example, in the DAE~\eqref{eq:example2_eq}, the first, second and ninth equations have the dimension of current and others have the dimension of voltage.
Similarly, the first five variables $\tilde{\xi}_1, \ldots, \tilde{\xi}_5$ of~\eqref{eq:example2_eq} have the dimension of current and the last five variables $\tilde{\eta}_1, \ldots, \tilde{\eta}_5$ have the dimension of voltage.
Then the dimension of each nonzero entry $A_{i,j}(s)$ of $A(s)$ must be $X_iY_j^{-1}$ according to the principle of dimensional homogeneity.
An important physical observation here is that all the nonzero coefficients of entries in $Q(s)$ are naturally regarded as dimensionless because they typically represent coefficients of conservation laws.
In addition, since the indeterminate $s$ corresponds to the time derivative, its dimension is the inverse $\T^{-1}$ of the dimension $\T$ of time.
Thus if $Q_{i,j}(s) \neq 0$, then $Q_{i,j}(s)$ must be a monomial $Q_{i,j}(1) s^{d_{i,j}}$ of dimension $\T^{-d_{i,j}}$ with $d_{i,j} = \deg Q_{i,j}(s)$.
Let $\lambda_i, \mu_j \in \mathbb{Q}$ such that $X_i$ and $Y_j$ are decomposed as $X_i = \T^{\lambda_i} X_i'$ and $Y_j = \T^{\mu_j}Y_j'$, where $X_i'$ and $Y_j'$ are physical dimensions that are not relevant to $\T$ in a using system of measurement.
Now it holds $\T^{-d_{i,j}} = X_iY_j^{-1} = \T^{\lambda_i - \mu_j} X_i'Y_j'^{-1}$ for $i \in R$ and $j \in C$ with $Q_{i,j}(s) \neq 0$.
This implies $d_{i,j} = -\lambda_i + \mu_j$ and thus we have $Q_{i,j}(s) = Q_{i,j}(1)s^{-\lambda_i + \mu_j}$ for all $i \in R$ and $j \in C$.
This is equivalent to (MP-DC) if every $\lambda_i$ and $\mu_j$ are integral.
Even if not, we can take integral $(\lambda', \mu')$ satisfying~\eqref{eq:dimensional_consistency_1}~\cite[Theorem~2.2.35(2)]{Murota2000}.
See~\cite[Section~3]{Murota2000} for more detail.

As described above, $\lambda_i$ and $\mu_j$ can be taken as the exponents of $\T$ in the physical dimensions of the $i$th equation and the $j$th variable (if they are integral).
In fact, the value~\eqref{eq:lambda_and_mu_of_RLC} is taken from the DAE~\eqref{eq:example2_eq} in this way as the dimension of voltage is expressed as $\mathsf{L}^2 \mathsf{T}^{-3} \mathsf{M} \mathsf{I}^{-1}$ by the SI base units, where $\mathsf{L}, \mathsf{M}$ and $\mathsf{I}$ are dimensions of length, mass and current, respectively.

\subsection{Improved Algorithm}
\label{sec:improved_algorithm}

This section improves the matrix modification procedure in Phase~3 for DCLM-polynomial matrices preserving their dimensional consistency.

Let $A(s) = \binom{Q(s)}{T(s)}$ be a DCLM-polynomial matrix with $R_Q = \Row(Q)$, $R_T = \Row(T)$ and $C = \Col(A)$.
Let $\prn{p, q}$ be an optimal solution of $\D{A}$.
For an integer $k \in \setZ$, let
\begin{align} \label{def:Rk_Ck}
  R_k = \set{i \in R_Q}[p_i - \lambda_i = k],
  \quad
  C_k = \set{j \in C}[q_j - \mu_j = k].
\end{align}
If $Q_{i,j}(s) \neq 0$, then we have $c_{i,j} \leq q_j - p_i$ from the feasibility of $(p, q)$ and $c_{i,j} = \mu_j - \lambda_i$ by~\eqref{eq:dimensional_consistency_1}.
Hence $p_i - \lambda_i \leq q_j - \mu_j$ follows, which implies $i \in R_h$ if and only if $j \in C_k$ with $h \leq k$.
Thus, it holds $Q(s)[R_h, C_k]=O$ for integers $h, k \in \setZ$ with $h > k$.
Namely, $Q(s)$ forms a block triangular matrix.

Let $A^\# = \binom{Q^\#}{T^\#}$ denote the tight coefficient matrix with respect to $\prn{p, q}$.
From the definition of the tight coefficient matrix, $Q^\#$ forms a block diagonal matrix as
\begin{align}
  Q^\# = \begin{blockarray}{ccccccc}
    & \cdots & C_{-1} & C_0 & C_1 & C_2 & \cdots \\
    \begin{block}{c(cccccc)}
      \vdots & \ddots \\
      R_{-1} && Q^\#_{-1} \\
      R_{0}  &&& Q^\#_0 \\
      R_{1}  &&&& Q^\#_1 \\
      R_{2}  &&&&& Q^\#_2 \\
      \vdots &&&&&& \ddots \\
    \end{block}
  \end{blockarray}
  \,\, ,
\end{align}
where $Q^\#_k = Q^\#[R_k, C_k]$ for $k \in \setZ$, and empty blocks indicate zero submatrices.

Let $J^* \subseteq C$ be a minimizer of the rank identity~\eqref{eq:rank_identity_min} for $A^\#$.
Sorting rows in ascending order of $p$, the matrix modification process described in \cref{sec:matrix_modification} finds a nonsingular upper-triangular matrix $U$ such that
\begin{align} \label{eq:UQsharp}
  \rank UQ^\#[R_Q, J^*] = \trank UQ^\#[R_Q, J^*].
\end{align}
For a DCLM-polynomial matrix, supposing that rows in $R_k$ are sorted in ascending order of $p$, we find a nonsingular upper-triangular matrix $U_k$ such that
\begin{align} \label{eq:UkQk}
  \rank U_k Q_k^\#[R_k, C_k \cap J^*] = \trank U_k Q_k^\#[R_k, C_k \cap J^*]
\end{align}
for $k \in \setZ$.
Then $U = \blockdiag(\ldots, U_{-1}, U_0, U_1, U_2, \ldots)$ satisfies~\eqref{eq:UQsharp}, where $\blockdiag(B_1, B_2, \ldots, B_N)$ is a block diagonal matrix of diagonal blocks $B_1, B_2, \ldots, B_N$.

For $k \in \setZ$, let $P_k(s)$ be a diagonal polynomial matrix with $\Row(P_k) = \Col(P_k) = R_k$ whose $(i, i)$ entry is $s^{p_i}$ for each $i \in R_k$.
Let $D_p(s) = \blockdiag(\ldots, P_{-1}(s), P_0(s), P_1(s), P_2(s), \ldots)$.
Now the unimodular matrix $U_Q(s)$ defined in~\eqref{def:U_tilde} can be written as
\begin{align}
  U_Q(s)
  &= D_p^{-1}(s) \blockdiag(\ldots, U_{-1}, U_0, U_1, U_2, \ldots) D_p(s) \nonumber \\
  &= D_p^{-1}(s) \blockdiag(\ldots, U_{-1} P_{-1}(s), U_0 P_0(s),  U_1 P_1(s), U_2 P_2(s), \ldots). \label{eq:DC_UQ}
\end{align}
Then we update $A(s)$ into $\bar{A}(s) = \binom{U_Q(s)Q(s)}{T(s)}$ as written in~\eqref{eq:update_A_2}.

\begin{lemma}
  Let $A(s) = \binom{Q(s)}{T(s)}$ be an $n \times n$ DCLM-polynomial matrix.
  Then $\bar{A}(s) = \binom{U_Q(s)Q(s)}{T(s)}$ is also dimensionally consistent.
\end{lemma}
\begin{proof}
  Let $\lambda_1, \ldots, \lambda_{m_Q}$ and $\mu_1, \ldots, \mu_n$ defined in~\eqref{eq:dimensional_consistency_1} for $A(s)$, where $m_Q = \card{\Row(Q)}$.
  For $k \in \setZ$, let $R_k$ and $C_k$ defined in~\eqref{def:Rk_Ck}, and let $\Lambda_k(s)$ denote a diagonal polynomial matrix with $\Row(\Lambda_k) = \Col(\Lambda_k) = R_k$ whose $(i, i)$ entry is $s^{\lambda_i}$ for each $i \in R_k$, and $D_\mu(s) = \diag(s^{\mu_1}, \ldots, s^{\mu_n})$.
  Then the condition~\eqref{eq:dimensional_consistency_1} for dimensional consistency is written as
  \begin{align} \label{eq:DC_Q}
    Q(s) = \blockdiag(\ldots, \Lambda_{-1}^{-1}(s), \Lambda_0^{-1}(s), \Lambda_1^{-1}(s), \Lambda_2^{-1}(s), \ldots) Q(1) D_\mu(s).
  \end{align}
  Combining~\eqref{eq:DC_UQ} and~\eqref{eq:DC_Q}, we obtain
  \begin{align}
    U_Q(s) Q(s)
    &= P^{-1}(s) \blockdiag(\ldots, U_{-1} P_{-1}(s) \Lambda_{-1}^{-1}(s), U_0 P_0(s) \Lambda_0^{-1}(s),  U_1 P_1(s) \Lambda_1^{-1}(s), \ldots) Q(1) D_\mu(s) \nonumber \\
    &= P^{-1}(s) \blockdiag(\ldots, s^{-1} U_{-1} , U_0,  s U_1, s^2 U_2, \ldots) Q(1) D_\mu(s) \nonumber \\
    &= \blockdiag(\ldots, s^{-1} P_{-1}^{-1}(s), P_0^{-1}(s), sP_1^{-1}(s), s^2P_2^{-1}(s), \ldots) UQ(1)D_\mu(s) \label{eq:DC_tilde_Q},
  \end{align}
  where we used $P_k(s) \Lambda_k^{-1}(s) = s^k I$ for $k \in \setZ$.
  From~\eqref{eq:DC_tilde_Q}, $\bar{A}(s)$ is also dimensionally consistent.
\end{proof}

\subsection{Complexity Analysis}
For a DCLM-polynomial matrix $A(s)$, we can compute $\bar{A}(s) = U(s) A(s)$ only by one constant matrix multiplication $U Q(1)$ from~\eqref{eq:DC_tilde_Q}, whereas a general LM-polynomial matrix needs $\Order\prn{ln}$ multiplications.
This improves the total running time as follows.

\begin{theorem} \label{thm:DC_complexity_tightness}
  Let $A(s)$ be an $n \times n$ nonsingular DCLM-polynomial matrix and $l$ the maximum degree of an entry in $A(s)$.
  Then Algorithm for Tightness runs in $\Order\prn{l n^4 \log n}$ time.
\end{theorem}
\begin{proof}
  For each iteration of Phases~2 and~3, the computation of $\bar{A}(s)$ in Phase~3 can be done in $\Order\prn{n^\omega}$ time, where $2 < \omega \leq 3$ is the matrix multiplication exponent.
  The most expensive part is the nonsingularity checking for a tight coefficient matrix in Phase~2, which requires $\Order\prn{n^3 \log n}$ time~\cite{Cunningham1986, Murota1987}.
  Since the number of iterations of Phases~2 and~3 is at most $\tdegdet{A} - \degdet{A} \leq ln$, the running time of Phases~2 and~3 is $\Order\prn{l n^4 \log n}$.
  We can check that other processes run in $\Order\prn{ln^4 \log n}$ time as in the proof of \cref{thm:complexity_tightness}.
\end{proof}

\begin{theorem} \label{thm:DC_complexity_all}
  For a DAE~\eqref{eq:laplaced_mixed_dae} with $n \times n$ nonsingular DCM-polynomial coefficient matrix $A(s)$, our algorithm returns an equivalent DAE of index zero or one in $\Order\prn{l n^4 \log n}$ time, where $l$ is the maximum degree of entries in $A(s)$.
\end{theorem}
\begin{proof}
  We can easily check that the coefficient LM-polynomial matrix of the augmented DAE described in \cref{sec:mixed_to_LM} is also dimensionally consistent.
  Algorithm for Tightness runs in $\Order\prn{ln^4 \log n}$ time from \cref{thm:DC_complexity_tightness}.
  In addition, the MS-algorithm runs in $\Order\prn{n^4 \log n + ln^2}$ time as discussed in the proof of \cref{thm:complexity_all}.
  Thus the total running time is $\Order\prn{ln^4 \log n}$.
\end{proof}

\section{Examples}
\label{sec:example}

We give two examples below.
The first example is a simple index-4 DAE and the second example is the DAE~\eqref{eq:example2_eq} representing the electrical network shown in \cref{fig:RLC}.
Throughout the execution of our algorithm, it is emphasized that: (i) we only use combinatorial operations and numerical calculations over rational numbers (over integers in the following examples), and (ii) we do not reference values of physical quantities.

\subsection{Example of High-index DAE}
The first example is the following index-4 DAE
\begin{align} \label{eq:example1}
  \left\{\begin{aligned}
    \ddot{x}_1 - \dot{x}_1 + \ddot{x}_2 - \dot{x}_2 + x_4 &= f_1(t), \\
    \ddot{x}_1 + \ddot{x}_2 + x_3 &= f_2(t), \\
    \alpha_1 x_2 + \alpha_2 \ddot{x}_3 + \alpha_3 \dot{x}_4 &= f_3(t), \\
    \alpha_4 x_3 + \alpha_5 \dot{x}_4 &= f_4(t),
  \end{aligned}\right.
\end{align}
with independent parameters $\alpha_1, \ldots, \alpha_5$ and smooth functions $f_1, \ldots, f_4$.
The coefficient matrix $A(s) = \binom{Q(s)}{T(s)}$ corresponding to~\eqref{eq:example1} is an LM-polynomial matrix given by
\begin{align} \label{eq:example1_A}
  A(s) = \begin{pmatrix}
    s^2-s & s^2-s    &              & 1 \\
    s^2   & s^2      & 1            &   \\
          & \alpha_1 & \alpha_2 s^2 & \alpha_3 s \\
          &          & \alpha_4     & \alpha_5 s
  \end{pmatrix}.
\end{align}
The row sets $R_Q$ of $Q(s)$ and $R_T$ of $T(s)$ correspond to the first and last two rows in $A(s)$, respectively.
This polynomial matrix~\eqref{eq:example1_A} is not DCLM.
Since $\degdet{A} = \deg \prn{-\alpha_1 \alpha_5 s^3 -\alpha_1 \alpha_4 s^2 + \alpha_1 \alpha_5 s^2} = 3$ and $\tdegdet{A} = 7$, the MS-algorithm is not applicable to the DAE, which is shown in our algorithm.

Let us apply our algorithm to~\eqref{eq:example1_A}.
First, we find a dual optimal solution $p = \prn{0, 0, 0, 0}$ and $q = \prn{2, 2, 2, 1}$.
The corresponding tight coefficient matrix $A^\# = \binom{Q^\#}{T^\#}$ is
\begin{align}
  A^\# = \begin{pmatrix}
    1 & 1 &   & \\
    1 & 1 &   & \\
      &   & \alpha_2 & \alpha_3 \\
      &   &          & \alpha_5 \\
  \end{pmatrix}.
\end{align}
A minimizer $J^*$ of~\eqref{eq:rank_identity_min} for $A^\#$ is the set of the first and the second columns.
Then we can check that $\rank A^\# = Q^\#[R_Q, J^*] + T^\#[R_T, J^*] + \card{C \setminus J^*} = 1 + 0 + 2 = 3 < 4$, which implies that $A(s)$ is not upper-tight.
We convert $Q^\#[R_Q, J^*] = \begin{pmatrix} 1 & 1 \\ 1 & 1 \end{pmatrix}$ by a row transformation into
$
  \bar{Q}^\#[R_Q, J^*] = U Q^\#[R_Q, J^*] = \begin{pmatrix} & \\ 1 & 1 \end{pmatrix},
$
where $U = \begin{pmatrix} 1 & -1 \\ & 1 \end{pmatrix}$.
Using $U_Q(s) = U$, the LM-polynomial matrix $A(s)$ is modified to
\begin{align}
  A'(s) =
  \begin{pmatrix}
    1 & -1 &   &  \\
      &  1 &   &  \\
      &    & 1 &  \\
      &    &   & 1
  \end{pmatrix} 
  A(s)
  =
  \begin{pmatrix}
    -s  & -s       & -1           & 1 \\
    s^2 & s^2      &  1           &   \\
        & \alpha_1 & \alpha_2 s^2 & \alpha_3 s \\
        &          & \alpha_4     & \alpha_5 s
  \end{pmatrix}.
\end{align}
The dual solution is updated to $p' = \prn{1, 0, 0, 1}$ and $q' = \prn{2, 2, 2, 2}$, and the corresponding tight coefficient matrix $A'^\# = \binom{Q'^\#}{T'^\#}$ of $A'(s)$ is
\begin{align}
  A'^\# = \begin{pmatrix}
    -1 & -1 &    & \\
     1 &  1 &    & \\
       &    & \alpha_2 & \\
       &    &          & \alpha_5 \\
  \end{pmatrix}.
\end{align}
The minimizer $J^*$ that we used above also minimizes the right-hand side of the rank identity~\eqref{eq:rank_identity_min} for $A'^\#$.
Since $A'^\#$ is still singular, we continue the modification process.
Noting the order of rows, we transform $Q'^\#[R_Q, J^*] = \begin{pmatrix} -1 & -1 \\ 1 & 1 \end{pmatrix}$ by $U' = \begin{pmatrix} 1 & \\ 1 & 1 \end{pmatrix}$ into
\begin{align}
  \bar{Q}'^\#[R_Q, J^*] = U' Q'^\#[R_Q, J^*] = \begin{pmatrix} -1 & -1 \\ & \end{pmatrix}.
\end{align}
We have $U_Q'(s) = \diag(s^{-1}, 1) U' \diag(s, 1) = \begin{pmatrix} 1 & \\ s & 1 \end{pmatrix}$, and modify $A'(s)$ to
\begin{align}
  A''(s) =
  \begin{pmatrix}
    1 &   &   &  \\
    s & 1 &   &  \\
      &   & 1 &  \\
      &   &   & 1
  \end{pmatrix} 
  A'(s)
  =
  \begin{pmatrix}
    -s & -s       & -1           & 1   \\
       &          & -s+1         & s   \\
       & \alpha_1 & \alpha_2 s^2 & \alpha_3 s \\
       &          & \alpha_4     & \alpha_5 s
  \end{pmatrix}.
\end{align}
The dual solution is updated to $p'' = \prn{1, 3, 2, 3}$ and $q'' = \prn{2, 2, 4, 4}$.
Our algorithm halts at this point since $A''(s)$ is upper-tight, which can be checked through the nonsingularity of the tight coefficient matrix $A''^\#$ again.
Now $\degdet{A}$ is computed as $\degdet{A} = \degdet{A''} = \tdegdet{A''} = 3$.
The resulting DAE is 
\begin{align} \label{eq:example1_result}
  \left\{\begin{aligned}
    -\dot{x}_1 -\dot{x}_2 -x_3 + x_4 &= f_1(t) - f_2(t), \\
    -\dot{x}_3 + x_3 + \dot{x}_4 &= \dot{f}_1(t) - \dot{f}_2(t) + f_2(t),\\
     \alpha_1 x_2 + \alpha_2 \ddot{x}_3 + \alpha_3 \dot{x}_4 &= f_3(t), \\
     \alpha_4 x_3 + \alpha_5 \dot{x}_4 &= f_4(t),
  \end{aligned}\right.
\end{align}
which is index two.

An index-1 DAE is obtained by applying the MS-algorithm to the DAE~\eqref{eq:example1_result}.
Instead of $\prn{p'', q''}$, we now use an optimal solution $\tilde{p} = \prn{0,2,1,2}$ and $\tilde{q} = \prn{1,1,3,3}$ of $\D{A''}$ to decrease the number of dummy variables as described in the proof of \cref{thm:complexity_all}.
Then the MS-algorithm outputs the index-1 DAE
\begin{align}
  \left\{\begin{aligned}
    -\dot{x}_1 -z^{[1]}_2 -x_3 + x_4 &= f_1(t) - f_2(t), \\
    -\dot{x}_3 + x_3 + \dot{x}_4 &= \dot{f}_1(t) - \dot{f}_2(t) + f_2(t),\\
    -z^{[2]}_3 + x_3 + z^{[2]}_4 &= \ddot{f}_1(t) - \ddot{f}_2(t) + \dot{f}_2(t),\\
    -z^{[3]}_3 + x_3 + z^{[3]}_4 &= \dddot{f}_1(t) - \dddot{f}_2(t) + \ddot{f}_2(t),\\
     \alpha_1 x_2 + \alpha_2 x_3 + \alpha_3 \dot{x}_4 &= f_3(t), \\
     \alpha_1 z^{[1]}_2 + \alpha_2 \dot{x}_3 + \alpha_3 z^{[2]}_4 &= \dot{f}_3(t), \\
     \alpha_4 x_3 + \alpha_5 \dot{x}_4 &= f_4(t), \\
     \alpha_4 \dot{x}_3 + \alpha_5 z^{[2]}_4 &= \dot{f}_4(t), \\
     \alpha_4 z^{[2]}_3 + \alpha_5 z^{[3]}_4 &= \ddot{f}_4(t), \\
  \end{aligned}\right.
\end{align}
where $z^{[1]}_2$, $z^{[2]}_3$, $z^{[3]}_3$, $z^{[2]}_4$ and $z^{[3]}_4$ are dummy variables corresponding to $\dot{x}_2$, $\ddot{x}_3$, $\dddot{x}_3$, $\ddot{x}_4$ and $\dddot{x}_4$, respectively.

\subsection{Example of Electrical Network}
\label{sec:example_2}
The next example is the DAE~\eqref{eq:example2_eq} representing the electrical network in~\cref{fig:RLC}.
Since the coefficient matrix $A(s)$ is not LM-polynomial, it seems that we cannot directly apply our algorithm to $A(s)$.
However, since each of the last five rows in $A(s)$ do not contain two or more accurate constants, we can convert $A(s)$ into an LM-polynomial matrix by multiplying an independent parameter to each of the rows.
In addition, by the same logic to \cref{lem:B}, our algorithm works without actually multiplying the independent parameters by regarding nonzero entries in the last five rows as independent parameters.
Thus we see $A(s)$ as an LM-polynomial matrix $A(s) = \binom{Q(s)}{T(s)}$, where $Q(s)$ and $T(s)$ correspond to the first and last five rows in $A(s)$, respectively.
The matrix $A(s)$ meets the condition (MP-DC) for DCLM-polynomial matrices with $\lambda = \prn{0, 0, 0, 0, 0}$ and $\mu = \prn{0, 0, 0, 0, 0, 0, 0, 0, 0, 0}$.

We are now ready for applying our algorithm to $A(s)$.
In Phase 1, a dual optimal solution is obtained as $p = \prn{0, 0, 0, 0, 0, 0, 0, 0, 0, 0}$ and $q = \prn{0, 0, 1, 0, 0, 0, 0, 0, 0, 1, 0}$, which implies that $\tdegdet{A} = 2$.
The corresponding tight coefficient matrix $A^\# = \binom{Q^\#}{T^\#}$ is given by
\begin{align}
  A^\# = 
  \prn{\begin{array}{rrrrr|rrrrr}
    -1  &     &     & -1 &  1 &    &    &    &    &    \\
        &   1 &     &  1 & -1 &    &    &    &    &    \\
    \hline
        &     &     &    &    &  1 &    &  1 &    & -1 \\
        &     &     &    &    & -1 & -1 &    &    &    \\
        &     &     &    &    &    &  1 & -1 &    &    \\
    \hline
    R_1 &     &     &    &    & -1 &    &    &    &    \\
        & R_2 &     &    &    &    & -1 &    &    &    \\
        &     &   L &    &    &    &    & -1 &    &    \\
        &     &     & -1 &    &    &    &    &  C &    \\
        &     &     &    &    &    &    &    &    &  1  
  \end{array}}
  \begin{array}{@{}l}
    \rdelim\}{5}{0em}[\text{$\scriptstyle{R_Q}$}]  \\
    \\
    \\
    \\
    \\
    \rdelim\}{5}{0em}[\text{$\scriptstyle{R_T}$}]  \\
    \\
    \\
    \\
    \\
  \end{array}
  \quad .
\end{align}
A minimizer $J^*$ of the rank identity~\eqref{eq:rank_identity_min} for $A^\#$ is the set of nine columns other than the rightmost column corresponding to the variable $\tilde{\eta}_5$.
Thus we can check
\begin{align}
  \rank A^\# = Q^\#[R_Q, J^*] + T^\#[R_T, J^*] + \card{C \setminus J^*} = 4 + 4 + 1 = 9 < 10,
\end{align}
which implies that $A(s)$ is not upper-tight.
We proceed to the matrix modification process for DCLM-polynomial matrices that we described in \cref{sec:improved_algorithm}.

The row set $R_k$ and the column set $C_k$ for $k \in \setZ$ defined in~\eqref{def:Rk_Ck} is the following:
\begin{align}
  \begin{blockarray}{ccc@{}c@{}ccccc@{}c@{\,}cc}
    &\overmat{9}{-1&-1&-1&-1&-1&-1-1&-1&-1}{J^*} \\
    &\overmat{2}{-1&-1}{C_0} & \overbrace{}^{C_1} & \overmat{5}{-1&-1&-1&-1&-1}{C_0} & \overbrace{}^{C_1} & \overbrace{}^{C_0}\\
    \begin{block}{c(cc@{}c@{}ccccc@{}c@{\,}c)c}
      &-1  &     &     & -1 &  1 &    &    &    &   &    & \rightmat{5}{R_0 .} \\
      &    &   1 &     &  1 & -1 &    &    &    &   &    \\
      Q^\# = &    &     &     &    &    &  1 &    &  1 &   & -1 \\
      &    &     &     &    &    & -1 & -1 &    &   &    \\
      &    &     &     &    &    &    &  1 & -1 &   &    \\
    \end{block}
  \end{blockarray}
\end{align}
Now $Q^\#$ can be seen as a block diagonal matrix consisting of one diagonal block $Q^\#_0 = Q^\#[R_0, C_0]$ by $Q^\#[R_0, C_1] = O$.
We transform
\begin{align}
  Q_0^\#[R_0, C_0 \cap J^*]
  =
  \prn{\begin{array}{rrrrrrr}
    -1  &     &  -1 &  1 &    &    &     \\
        &   1 &   1 & -1 &    &    &     \\
        &     &     &    &  1 &    &   1 \\
        &     &     &    & -1 & -1 &     \\
        &     &     &    &    &  1 &  -1 \\
  \end{array}}
\end{align}
into $
  U Q_0^\#[R_0, C_0 \cap J^*] 
  =
  \prn{\begin{array}{rrrrrrr}
    -1  &     &  -1 &  1 &    &    &     \\
        &   1 &   1 & -1 &    &    &     \\
        &     &     &    &    &    &     \\
        &     &     &    & -1 & -1 &     \\
        &     &     &    &    &  1 &  -1 \\
  \end{array}}$, where $U = \prn{\begin{array}{ccccc} 1 &&&& \\ & 1 &&& \\ && 1 & 1 & 1 \\ &&& 1 & \\ &&&& 1 \end{array}}$.
Using $U_Q(s) = U$, we modify $A(s)$ to
\begin{align}
  A'(s) = 
  \prn{\begin{array}{rrrrr|rrrrr}
    -1  &     &     & -1 &  1 &    &    &    &    &    \\
        &   1 &   1 &  1 & -1 &    &    &    &    &    \\
    \hline
        &     &     &    &    &    &    &    &  1 & -1 \\
        &     &     &    &    & -1 & -1 &    &  1 &    \\
        &     &     &    &    &    &  1 & -1 &    &    \\
    \hline
    R_1 &     &     &    &    & -1 &    &    &    &    \\
        & R_2 &     &    &    &    & -1 &    &    &    \\
        &     &  Ls &    &    &    &    & -1 &    &    \\
        &     &     & -1 &    &    &    &    & Cs &    \\
        &     &     &    &    &    &    &    &    &  1  
  \end{array}},
\end{align}
where the third row is different between $A(s)$ and $A'(s)$.
The dual solution is updated to $p' = \prn{0, 0, 1, 0, 0, 0, 0, 0, 0, 1}$ and $q' = \prn{0, 0, 1, 0, 0, 0, 0, 0, 1, 1}$.
Since the corresponding tight coefficient matrix of $A'(s)$ is nonsingular, we stop the algorithm.
The index of the modified DAE remains at two.

Finally, by applying the MS-algorithm to the modified DAE, we obtain an index-1 DAE
\begin{align}
  \left\{\begin{aligned}
    -\xi_1 -\xi_4 +\xi_5 &= 0, \\
    \xi_2 + \xi_3 + \xi_4 - \xi_5 &= 0, \\
    \eta_4 - \eta_5 &= 0, \\
    z^{[1]}_4 - z^{[1]}_5 &= 0, \\
    -\eta_1 -\eta_2 + \eta_4 &= 0, \\
    \eta_2 - \eta_3 &= 0, \\
    R_1 \xi_1 - \eta_1 &= 0, \\
    R_2 \xi_2 - \eta_2 &= 0, \\
    L \dot{\xi}_3 - \eta_3 &= 0, \\
    -\xi_4 + C z^{[1]}_4 &= 0, \\
    \eta_5 &= V(t), \\
    z^{[1]}_5 &= \dot{V}(t),
  \end{aligned}\right.
\end{align}
where $z^{[1]}_4$ and $z^{[1]}_5$ are dummy variables corresponding to $\dot{\eta}_4$ and $\dot{\eta}_5$, respectively.

\section{Numerical Experiments}
\label{sec:experiments}

We conduct numerical experiments comparing our algorithm with the LC-method by Tan et al.~\cite{Tan2017}.
Recall that the LC-method works for linear DAEs whose associated polynomial matrix $A(s)$ has only constants, whereas our algorithm can treat a DAE containing independent parameters.

\subsection{Experiment Description}
For an even positive integer $K$, the Butterworth filter via the $K$-th Cauer topology is an electrical circuit shown in \cref{fig:cauer}.
This circuit has $n = 2K+4$ state variables $\xi_0, \xi_1, \ldots, \xi_{K+1}, \eta_0, \eta_1, \ldots, \eta_{K+1}$, where $\xi_0, \xi_1, \ldots, \xi_{K+1}$ are currents and $\eta_0, \eta_1, \ldots, \eta_{K+1}$ are voltages shown in \cref{fig:cauer}.

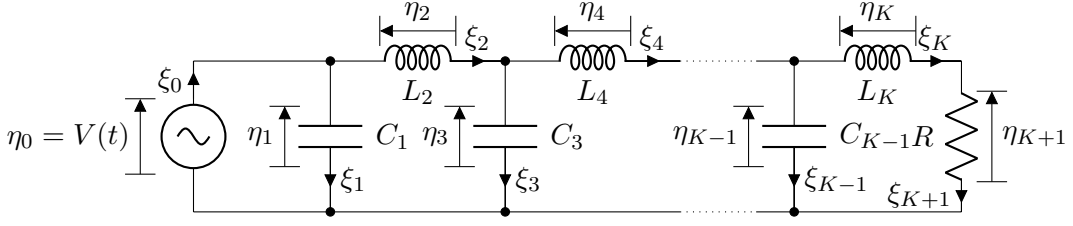
\begin{figure}[tbp]
  \centering
  \begin{circuitikz}
    \draw (6.4,-1)
      to[short] (0,-1)
      to[short] (0,-0.5)
      to[sV] (0,0.5)
      to[short, i=$\xi_0$] (0,1)
      to[short] (1.8,1)
      to[C=$C_1$, i=$\xi_1$, *-*] (1.8,-1);
    \draw[-{Triangle[angle=50:1.9mm]}] (-0.7, -0.5) -- (-0.7, 0.5) node[left, midway] {$\eta_0 = V(t)$};
    \draw (-0.9, -0.5) -- (-0.5, -0.5) {};
    \draw (-0.9, 0.5) -- (-0.5, 0.5) {};
    \draw[-{Triangle[angle=50:1.9mm]}] (1.2, -0.4) -- (1.2, 0.4) node[left, midway] {$\eta_1$};
    \draw (1.0, -0.4) -- (1.4, -0.4) {};
    \draw (1.0, 0.4) -- (1.4, 0.4) {};
    \draw (1.8,1)
      to[L, l_=$L_2$, i=$\xi_2$] (4.1,1)
      to[C=$C_3$, i=$\xi_3$, *-*] (4.1,-1);
    \draw[-{Triangle[angle=50:1.9mm]}] (3.5, -0.4) -- (3.5, 0.4) node[left, midway] {$\eta_3$};
    \draw (3.3, -0.4) -- (3.7, -0.4) {};
    \draw (3.3, 0.4) -- (3.7, 0.4) {};
    \draw[-{Triangle[angle=50:1.9mm]}] (3.45, 1.4) -- (2.45, 1.4) node[above, midway] {$\eta_2$};
    \draw (3.45, 1.2) -- (3.45, 1.6) {};
    \draw (2.45, 1.2) -- (2.45, 1.6) {};
    \draw (4.1,1)
      to[L, l_=$L_4$, i=$\xi_4$] (6.4,1);
    \draw[-{Triangle[angle=50:1.9mm]}] (5.75, 1.4) -- (4.75, 1.4) node[above, midway] {$\eta_4$};
    \draw (5.75, 1.2) -- (5.75, 1.6) {};
    \draw (4.75, 1.2) -- (4.75, 1.6) {};
    \draw [dotted] (6.4,1) -- (7.4,1);
    \draw [dotted] (6.4,-1) -- (7.4,-1);
    \draw (7.4,1)
      to[short] (7.9,1)
      to[C=$C_{K-1}$, i=$\xi_{K-1}$, *-*] (7.9,-1);
    \draw[-{Triangle[angle=50:1.9mm]}] (7.3, -0.4) -- (7.3, 0.4) node[left, midway] {$\eta_{K-1}$};
    \draw (7.1, -0.4) -- (7.5, -0.4) {};
    \draw (7.1, 0.4) -- (7.5, 0.4) {};
    \draw (7.9,1)
      to[L, l_=$L_K$, i=$\xi_K$] (10.1,1)
      to[R, l_=$R$, i_=$\xi_{K+1}$] (10.1,-1)
      to[short] (7.4,-1);
    \draw[-{Triangle[angle=50:1.9mm]}] (9.5, 1.4) -- (8.5, 1.4) node[above, midway] {$\eta_K$};
    \draw (9.5, 1.2) -- (9.5, 1.6) {};
    \draw (8.5, 1.2) -- (8.5, 1.6) {};
    \draw[-{Triangle[angle=50:1.9mm]}] (10.5, -0.6) -- (10.5, 0.6) node[right, midway] {$\eta_{K+1}$};
    \draw (10.3, -0.6) -- (10.7, -0.6) {};
    \draw (10.3, 0.6) -- (10.7, 0.6) {};
  \end{circuitikz}
  \caption{%
    Butterworth filter via the $K$-th Cauer topology.
  }
  \label{fig:cauer}
\end{figure}

A DAE representing the circuit is given by
\begin{align} \label{eq:cauer}
  \left\{\begin{aligned}
    -\xi_{k-1} + \xi_{k} + \xi_{k+1} &= 0 \quad (k=1, 3, 5, \ldots, K-1), \\
    -\xi_0 + \xi_1 + \xi_3 + \cdots + \xi_{K+1} &= 0, \\
    \eta_0 + \eta_2 + \eta_4 + \cdots + \eta_K + \eta_{K+1} &= 0, \\
    -\eta_{k-1} + \eta_{k} + \eta_{k+1} &= 0 \quad (k=2, 4, 6, \ldots, K), \\
    \eta_0 &= V(t), \\
    -\xi_k + C_k \dot{\eta}_k &= 0 \quad (k=1, 3, 5, \ldots, K-1), \\
    L_k \dot{\xi}_k - \eta_k &= 0 \quad (k=2, 4, 6, \ldots, K), \\
    R \xi_{K+1} - \eta_{K+1} &= 0.
  \end{aligned}\right.
\end{align}
The index of the DAE~\eqref{eq:cauer} is two, and the associated polynomial matrix $A(s)$ is a sparse matrix which has $6K+7$ nonzero coefficients.
Though it suffices to use simpler equations $-\xi_{K} + \xi_{K+1} = 0$ and $\eta_0 + \eta_1 = 0$ instead of the second and the third equations in~\eqref{eq:cauer}, respectively, we use them to make the MS-algorithm not applicable for the DAE.

We apply our algorithm and the LC-method to the DAE~\eqref{eq:cauer} using the following two ways of implementations:

\begin{description}
  \item[Dense Matrix Implementation\normalfont{,}]
    which stores a matrix in the memory as a two-dimensional array.
    While this implementation always requires $\Order\prn{nm}$ space for a matrix of size $m \times n$, it has less overhead than the sparse matrix implementation if the matrix is dense.
    
  \item[Sparse Matrix Implementation\normalfont{,}]
    which stores only nonzero entries of a matrix.
    A typical implementation of this type is in formats called a \emph{compressed sparse column} (CSC) or a \emph{compressed sparse row} (CSR).
    We adopt the CSR in our experiments.
    The sparse matrix implementation has an advantage that it consumes only the space proportional to the number of nonzero entries, and thus algorithms using this implementation are expected to run efficiently for sparse matrices.
\end{description}

In our algorithm, we treat the coefficients $R$, $C_k$, $L_k$ and `$\pm1$'s in the last four equations in~\eqref{eq:cauer} as independent parameters similarly to the example in \cref{sec:example_2}.
Then the associated polynomial matrix $A(s) = \binom{Q(s)}{T(s)}$ is dimensionally consistent, where $\card{\Row(Q(s))} = \card{\Row(T(s))} = K + 2$.
In the LC-method, we substitute the following real numbers:
\begin{align}
  C_k &= 2 \sin \frac{2k-1}{2K} \pi \quad (k = 1, 3, 5, \ldots, K-1), \\
  L_k &= 2 \sin \frac{2k-1}{2K} \pi \quad (k = 2, 4, 6, \ldots, K), \\
  R   &= \pi.
\end{align}
Under this setting, we compare the running time for $K = 2, 4, 8, 16, \ldots, 2^{16}$.
We implemented all algorithms in \CC\ using the library \texttt{Eigen3} for matrix computation.
It is emphasized again that we do not rely on symbolic computation.
The experiments are conducted on a laptop with Core~i7 \SI{1.7}{GHz} CPU and \SI{8}{GB} memory.

\subsection{Experimental Results}

\paragraph{Running time.}

\begin{figure}[tbp]
  \centering
  \begin{tikzpicture}[gnuplot]
\tikzset{every node/.append style={scale=0.80}}
\path (0.000,0.000) rectangle (10.500,7.000);
\gpcolor{color=gp lt color border}
\gpsetlinetype{gp lt border}
\gpsetdashtype{gp dt solid}
\gpsetlinewidth{1.00}
\draw[gp path] (1.495,0.787)--(1.675,0.787);
\draw[gp path] (10.058,0.787)--(9.878,0.787);
\node[gp node right] at (1.348,0.787) {$\SI{100}{\micro\second}$};
\draw[gp path] (1.495,1.533)--(1.675,1.533);
\draw[gp path] (10.058,1.533)--(9.878,1.533);
\node[gp node right] at (1.348,1.533) {$\SI{1}{\milli\second}$};
\draw[gp path] (1.495,2.279)--(1.675,2.279);
\draw[gp path] (10.058,2.279)--(9.878,2.279);
\node[gp node right] at (1.348,2.279) {$\SI{10}{\milli\second}$};
\draw[gp path] (1.495,3.024)--(1.675,3.024);
\draw[gp path] (10.058,3.024)--(9.878,3.024);
\node[gp node right] at (1.348,3.024) {$\SI{100}{\milli\second}$};
\draw[gp path] (1.495,3.770)--(1.675,3.770);
\draw[gp path] (10.058,3.770)--(9.878,3.770);
\node[gp node right] at (1.348,3.770) {$\SI{1}{\second}$};
\draw[gp path] (1.495,4.516)--(1.675,4.516);
\draw[gp path] (10.058,4.516)--(9.878,4.516);
\node[gp node right] at (1.348,4.516) {$\SI{10}{\second}$};
\draw[gp path] (1.495,5.262)--(1.675,5.262);
\draw[gp path] (10.058,5.262)--(9.878,5.262);
\node[gp node right] at (1.348,5.262) {$\SI{e2}{\second}$};
\draw[gp path] (1.495,6.007)--(1.675,6.007);
\draw[gp path] (10.058,6.007)--(9.878,6.007);
\node[gp node right] at (1.348,6.007) {$\SI{e3}{\second}$};
\draw[gp path] (1.495,6.753)--(1.675,6.753);
\draw[gp path] (10.058,6.753)--(9.878,6.753);
\node[gp node right] at (1.348,6.753) {$\SI{e4}{\second}$};
\draw[gp path] (2.066,0.787)--(2.066,0.967);
\draw[gp path] (2.066,6.753)--(2.066,6.573);
\node[gp node center] at (2.066,0.541) {$2^{2}$};
\draw[gp path] (3.208,0.787)--(3.208,0.967);
\draw[gp path] (3.208,6.753)--(3.208,6.573);
\node[gp node center] at (3.208,0.541) {$2^{4}$};
\draw[gp path] (4.349,0.787)--(4.349,0.967);
\draw[gp path] (4.349,6.753)--(4.349,6.573);
\node[gp node center] at (4.349,0.541) {$2^{6}$};
\draw[gp path] (5.491,0.787)--(5.491,0.967);
\draw[gp path] (5.491,6.753)--(5.491,6.573);
\node[gp node center] at (5.491,0.541) {$2^{8}$};
\draw[gp path] (6.633,0.787)--(6.633,0.967);
\draw[gp path] (6.633,6.753)--(6.633,6.573);
\node[gp node center] at (6.633,0.541) {$2^{10}$};
\draw[gp path] (7.775,0.787)--(7.775,0.967);
\draw[gp path] (7.775,6.753)--(7.775,6.573);
\node[gp node center] at (7.775,0.541) {$2^{12}$};
\draw[gp path] (8.916,0.787)--(8.916,0.967);
\draw[gp path] (8.916,6.753)--(8.916,6.573);
\node[gp node center] at (8.916,0.541) {$2^{14}$};
\draw[gp path] (10.058,0.787)--(10.058,0.967);
\draw[gp path] (10.058,6.753)--(10.058,6.573);
\node[gp node center] at (10.058,0.541) {$2^{16}$};
\draw[gp path] (1.495,6.753)--(1.495,0.787)--(10.058,0.787)--(10.058,6.753)--cycle;
\node[gp node center,rotate=-270] at (0.343,3.770) {running time};
\node[gp node center] at (5.776,0.172) {$K$};
\node[gp node right] at (4.288,6.450) {LC-method (dense)};
\gpcolor{rgb color={0.000,0.000,0.000}}
\gpsetdashtype{dash pattern=on 2.00*\gpdashlength off 3.00*\gpdashlength }
\gpsetlinewidth{2.00}
\draw[gp path] (4.435,6.450)--(5.203,6.450);
\draw[gp path] (1.495,0.923)--(2.066,1.111)--(2.637,1.354)--(3.208,1.764)--(3.778,2.409)%
  --(4.349,2.989)--(4.920,3.550)--(5.491,4.162)--(6.062,4.811)--(6.633,5.525)--(7.204,6.240);
\gpsetpointsize{6.00}
\gppoint{gp mark 9}{(1.495,0.923)}
\gppoint{gp mark 9}{(2.066,1.111)}
\gppoint{gp mark 9}{(2.637,1.354)}
\gppoint{gp mark 9}{(3.208,1.764)}
\gppoint{gp mark 9}{(3.778,2.409)}
\gppoint{gp mark 9}{(4.349,2.989)}
\gppoint{gp mark 9}{(4.920,3.550)}
\gppoint{gp mark 9}{(5.491,4.162)}
\gppoint{gp mark 9}{(6.062,4.811)}
\gppoint{gp mark 9}{(6.633,5.525)}
\gppoint{gp mark 9}{(7.204,6.240)}
\gppoint{gp mark 9}{(4.819,6.450)}
\gpcolor{color=gp lt color border}
\node[gp node right] at (4.288,6.204) {proposed (dense)};
\gpcolor{rgb color={0.000,0.000,0.000}}
\gpsetdashtype{gp dt solid}
\draw[gp path] (4.435,6.204)--(5.203,6.204);
\draw[gp path] (1.495,0.799)--(2.066,0.991)--(2.637,1.213)--(3.208,1.551)--(3.778,2.121)%
  --(4.349,2.643)--(4.920,3.194)--(5.491,3.802)--(6.062,4.426)--(6.633,5.087)--(7.204,5.762);
\gppoint{gp mark 7}{(1.495,0.799)}
\gppoint{gp mark 7}{(2.066,0.991)}
\gppoint{gp mark 7}{(2.637,1.213)}
\gppoint{gp mark 7}{(3.208,1.551)}
\gppoint{gp mark 7}{(3.778,2.121)}
\gppoint{gp mark 7}{(4.349,2.643)}
\gppoint{gp mark 7}{(4.920,3.194)}
\gppoint{gp mark 7}{(5.491,3.802)}
\gppoint{gp mark 7}{(6.062,4.426)}
\gppoint{gp mark 7}{(6.633,5.087)}
\gppoint{gp mark 7}{(7.204,5.762)}
\gppoint{gp mark 7}{(4.819,6.204)}
\gpcolor{color=gp lt color border}
\node[gp node right] at (4.288,5.958) {LC-method (sparse)};
\gpcolor{rgb color={0.000,0.000,0.000}}
\gpsetdashtype{dash pattern=on 2.00*\gpdashlength off 2.00*\gpdashlength }
\gpsetlinewidth{1.00}
\draw[gp path] (4.435,5.958)--(5.203,5.958);
\draw[gp path] (1.495,0.990)--(2.066,0.974)--(2.637,1.050)--(3.208,1.223)--(3.778,1.389)%
  --(4.349,1.618)--(4.920,1.953)--(5.491,2.302)--(6.062,2.608)--(6.633,2.999)--(7.204,3.398)%
  --(7.775,3.812)--(8.345,4.248)--(8.916,4.696)--(9.487,5.152)--(10.058,5.617);
\gppoint{gp mark 8}{(1.495,0.990)}
\gppoint{gp mark 8}{(2.066,0.974)}
\gppoint{gp mark 8}{(2.637,1.050)}
\gppoint{gp mark 8}{(3.208,1.223)}
\gppoint{gp mark 8}{(3.778,1.389)}
\gppoint{gp mark 8}{(4.349,1.618)}
\gppoint{gp mark 8}{(4.920,1.953)}
\gppoint{gp mark 8}{(5.491,2.302)}
\gppoint{gp mark 8}{(6.062,2.608)}
\gppoint{gp mark 8}{(6.633,2.999)}
\gppoint{gp mark 8}{(7.204,3.398)}
\gppoint{gp mark 8}{(7.775,3.812)}
\gppoint{gp mark 8}{(8.345,4.248)}
\gppoint{gp mark 8}{(8.916,4.696)}
\gppoint{gp mark 8}{(9.487,5.152)}
\gppoint{gp mark 8}{(10.058,5.617)}
\gppoint{gp mark 8}{(4.819,5.958)}
\gpcolor{color=gp lt color border}
\node[gp node right] at (4.288,5.712) {proposed (sparse)};
\gpcolor{rgb color={0.000,0.000,0.000}}
\gpsetdashtype{gp dt solid}
\draw[gp path] (4.435,5.712)--(5.203,5.712);
\draw[gp path] (1.495,0.819)--(2.066,0.898)--(2.637,1.045)--(3.208,1.200)--(3.778,1.401)%
  --(4.349,1.683)--(4.920,1.982)--(5.491,2.317)--(6.062,2.744)--(6.633,3.123)--(7.204,3.531)%
  --(7.775,3.969)--(8.345,4.418)--(8.916,4.875)--(9.487,5.366)--(10.058,5.878);
\gppoint{gp mark 6}{(1.495,0.819)}
\gppoint{gp mark 6}{(2.066,0.898)}
\gppoint{gp mark 6}{(2.637,1.045)}
\gppoint{gp mark 6}{(3.208,1.200)}
\gppoint{gp mark 6}{(3.778,1.401)}
\gppoint{gp mark 6}{(4.349,1.683)}
\gppoint{gp mark 6}{(4.920,1.982)}
\gppoint{gp mark 6}{(5.491,2.317)}
\gppoint{gp mark 6}{(6.062,2.744)}
\gppoint{gp mark 6}{(6.633,3.123)}
\gppoint{gp mark 6}{(7.204,3.531)}
\gppoint{gp mark 6}{(7.775,3.969)}
\gppoint{gp mark 6}{(8.345,4.418)}
\gppoint{gp mark 6}{(8.916,4.875)}
\gppoint{gp mark 6}{(9.487,5.366)}
\gppoint{gp mark 6}{(10.058,5.878)}
\gppoint{gp mark 6}{(4.819,5.712)}
\gpcolor{color=gp lt color border}
\draw[gp path] (1.495,6.753)--(1.495,0.787)--(10.058,0.787)--(10.058,6.753)--cycle;
\gpdefrectangularnode{gp plot 1}{\pgfpoint{1.495cm}{0.787cm}}{\pgfpoint{10.058cm}{6.753cm}}
\end{tikzpicture}
  \caption{Log-log plot of the experimental result: $K$ versus the running time.}
  \label{fig:graph}
\end{figure}
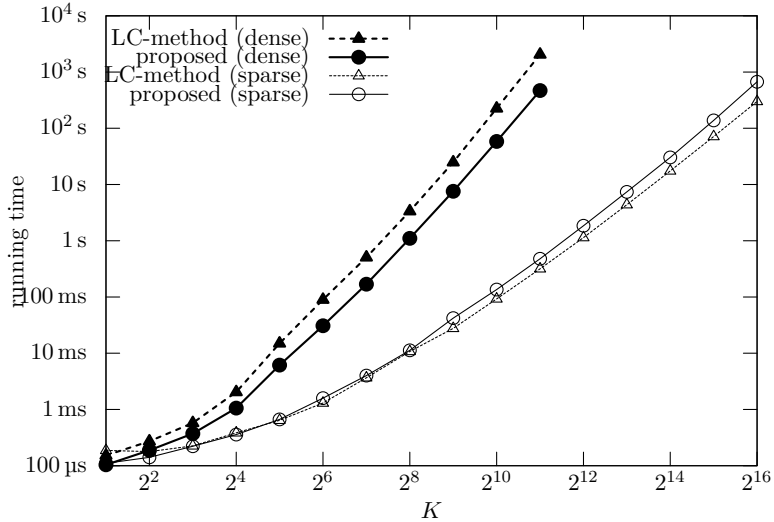

\begin{table}[tbp]
\centering
\caption{Running time (sec) of dense implementations for $K=2^{11}$.}
\label{tbl:time_dense}
\begin{tabular}{c|l@{\,}rl@{\,}r} \hline
             & \multicolumn{2}{c}{LC-method} & \multicolumn{2}{c}{proposed} \\ \hline
Phase~1      &  $1.80 \times 10^{-2}$ &   (0.00\%) & $1.70 \times 10^{-2}$ &  (0.00\%)  \\
Phase~2      &  $6.69 \times 10^{ 2}$ &  (29.61\%) & $9.69 \times 10^{ 1}$ &  (19.54\%) \\
Phase~3      &  $1.59 \times 10^{ 3}$ &  (70.26\%) & $3.97 \times 10^{ 2}$ &  (79.98\%) \\
MS-algorithm &  $1.02 \times 10^0$    &   (0.04\%) & $7.28 \times 10^{-1}$ &   (0.15\%) \\ \hline
total        &  $2.26 \times 10^{3}$  & (100.00\%) & $4.96 \times 10^{ 2}$ & (100.00\%) \\ \hline
\end{tabular}
\end{table}

\begin{table}[tbp]
\centering
\caption{Running time (sec) of sparse implementations for $K=2^{11}$.}
\label{tbl:time_sparse}
\begin{tabular}{c|l@{\,}rl@{\,}r} \hline
             & \multicolumn{2}{c}{LC-method} & \multicolumn{2}{c}{proposed} \\ \hline
Phase~1      &  $1.55 \times 10^{-2}$ &   (4.88\%) & $1.58 \times 10^{-2}$ &   (3.32\%)  \\
Phase~2      &  $1.33 \times 10^{-1}$ &  (41.87\%) & $3.82 \times 10^{-1}$ &  (80.07\%) \\
Phase~3      &  $1.25 \times 10^{-1}$ &  (39.40\%) & $39.2 \times 10^{-2}$ &   (8.21\%) \\
MS-algorithm &  $2.54 \times 10^{-2}$ &   (7.98\%) & $2.47 \times 10^{-2}$ &   (5.17\%) \\ \hline
total        &  $3.18 \times 10^{-1}$ & (100.00\%) & $4.78 \times 10^{-1}$ & (100.00\%) \\ \hline
\end{tabular}
\end{table}

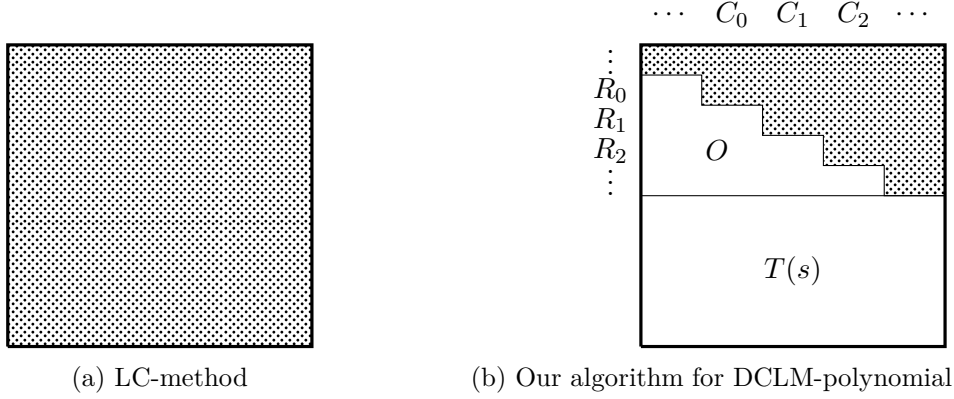
\begin{figure}[tbp]
  \begin{minipage}[b]{0.5\hsize}%
    \centering
    \begin{tikzpicture}[scale=0.8]
      \filldraw[very thick, pattern=crosshatch dots] (0,0) -- (5,0) -- (5,5) -- (0,5) -- (0,0);
    \end{tikzpicture}
    \subcaption{LC-method}
  \end{minipage}%
  \begin{minipage}[b]{0.5\hsize}
    \centering
    \begin{tikzpicture}[scale=0.8]
      \draw[very thick] (0,0) -- (5,0) -- (5,5) -- (0,5) -- (0,0);
      \draw (0, 2.5) -- (5, 2.5);
      \draw (0, 4.5) -- (1, 4.5) -- (1, 4) -- (2, 4) -- (2, 3.5) -- (3, 3.5) -- (3, 3) -- (4, 3) -- (4, 2.5) -- (5, 2.5);
      \fill[pattern=crosshatch dots] (0, 4.5) -- (1, 4.5) -- (1, 4) -- (2, 4) -- (2, 3.5) -- (3, 3.5) -- (3, 3) -- (4, 3) -- (4, 2.5) -- (5, 2.5) -- (5, 5) -- (0, 5) -- (0, 4.5);
      \node at (2.5, 1.25) {$T(s)$};
      \node at (1.25, 3.25) {$O$};
      \node at (-0.5, 4.85) {$\vdots$};
      \node at (-0.5, 4.25) {$R_0$};
      \node at (-0.5, 3.75) {$R_1$};
      \node at (-0.5, 3.25) {$R_2$};
      \node at (-0.5, 2.85) {$\vdots$};
      \node at (0.5, 5.5) {$\cdots$};
      \node at (1.5, 5.5) {$C_0$};
      \node at (2.5, 5.5) {$C_1$};
      \node at (3.5, 5.5) {$C_2$};
      \node at (4.5, 5.5) {$\cdots$};
    \end{tikzpicture}
    \subcaption{Our algorithm for DCLM-polynomial matrices}
  \end{minipage}%
  \caption{%
    Hatched regions indicate submatrices in a polynomial matrix $A(s)$ to be modified by algorithms.
    In (b), we use notations given in \cref{sec:improved_algorithm}.
  }
  \label{fig:regions}
\end{figure}

\Cref{tbl:time_dense,tbl:time_sparse} and \cref{fig:graph} show the running time of the algorithms.
On the dense matrix implementations, both algorithms did not run for $K \geq 2^{12}$ due to the lack of memory capacity.
The reasons are as follows.
Our implementations express a polynomial as an array of coefficients using \texttt{std::vector<int>} or \texttt{std::vector<float>} in \CC, and it consumes \SI{32}{bytes} even for the zero polynomial.
Since the number of entries in the input polynomial matrix $A(s)$ for $K = 2^{12}$ is $n^2 = (2K+4)^2 \geq 2^{26}$, we need at least $2^{26} \times \SI{32}{bytes} = \SI{2}{GB}$ to hold $A(s)$.
Besides the input matrix, our implementations construct several constant and polynomial matrices of similar or larger size, such as a tight coefficient matrix $A^\#$, a unimodular matrix $U(s)$ for modification in Phase~3 and an output matrix.
Thus $2^{12}$ is near the borderline of the maximum $K$ for which our implementations run on our laptop with $8$ GB memory.

It can be seen from \cref{fig:graph} that our algorithm is faster than the LC-method on their dense matrix implementations, and it is converse for their sparse ones.
This is attributed to the fact that in the process of multiplying polynomial matrices in~\eqref{eq:update_A_2} at Phase~3, the LC-method multiplies the entire of the given polynomial matrix $A(s)$ whereas our algorithm multiplies only submatrices of $A(s)$ as illustrated in \cref{fig:regions}.
Since this process is dominant on the dense matrix implementations as \cref{tbl:time_dense} indicates, the difference between the sizes of matrices to be multiplied directly affects the difference of the running times.
This process, however, does not cost much in the sparse matrix implementations, and thus Phase~2 becomes relatively expensive.
As a result, the difference between the running times on sparse matrix implementations reflect the difference between that of the independent matching algorithm and the Gaussian elimination used by our algorithm and the LC-method in Phase~2, respectively.

Recalling that the size of the DAE is $n = \Order\prn{K}$, \cref{fig:graph} shows that the running time of our algorithm grows proportionally to $\Order\prn{n^{2.84}}$ in the dense matrix implementation and $\Order\prn{n^{1.97}}$ in the sparse one for $K \geq 2^8$.
Both are much faster than the theoretical guarantee $\Order\prn{n^4 \log n}$ given in \cref{thm:DC_complexity_all}.

\section{Application to Nonlinear DAEs}
\label{sec:nonlinear_dae}

In this section, we discuss the application of our algorithm to nonlinear DAEs.
The $\sigma \nu$-method~\cite{Chowdhry2004}, which is implemented in Mathematica~\cite{mathematica}, adopts a strategy of treating nonlinear or time-varying terms as independent parameters in the Jacobian matrices of DAEs.
We first describe the $\sigma \nu$-method briefly.

Consider the index-2 nonlinear DAE~\eqref{def:nonlinear_dae} with a smooth nonlinear function $\funcdoms{g}{\setR}{\setR}$.
Their method constructs two kinds of Jacobian matrices JD and JV as follows:
\begin{align} 
  \mathrm{JD} = \prn{\pdif{F_i}{\dot{x}_j}}_{i,j} = \begin{pmatrix}
    1 & 0 & 0 \\
    1 & 0 & 0 \\
    1 & 0 & 0 
  \end{pmatrix}
  ,
  \quad
  \mathrm{JV} = \prn{\pdif{F_i}{x_j}}_{i,j} = \begin{pmatrix}
    0 & \partial g/\partial x_2 & 0 \\
    1 & 0 & 1 \\
    0 & 0 & 1 
  \end{pmatrix}.
\end{align}
If JD is nonsingular, the DAE is index zero from the implicit function theorem. 
Otherwise, the method performs Gaussian elimination on JD (and JV simultaneously) to make the bottom row of JD zero.
Then the method differentiates the equation corresponding to the bottom row, and checks the nonsingularity of JD again.
The main feature of the $\sigma \nu$-method is to treat nonlinear or time-varying terms as ``independent parameters'' to avoid complicated symbolic manipulations.
The method works according to the rule that arithmetic operations and the differentiation of independent parameters generate new independent parameters.

The $\sigma \nu$-method may fail due to this rule.
For example, let $\alpha_1$ be an independent parameter representing $\partial g/\partial x_2$ in JV.
By subtracting the first row from the second and third ones, we obtain
\begin{align}
  \mathrm{JD} = \begin{pmatrix}
    1 & 0 & 0 \\
    0 & 0 & 0 \\
    0 & 0 & 0 
  \end{pmatrix}
  ,
  \quad
  \mathrm{JV} = \begin{pmatrix}
    0 & \alpha_1 & 0 \\
    1 & \alpha_2 & 1 \\
    0 & \alpha_3 & 1 
  \end{pmatrix},
\end{align}
where $\alpha_2 = 0 - \alpha_1$ and $\alpha_3 = 0 - \alpha_1$ are newly generated parameters by the rule of arithmetic operations. 
We differentiate the second and third rows.
Then JD and JV are
\begin{align}
  \mathrm{JD} = \begin{pmatrix}
    1 & 0 & 0 \\
    1 & \alpha_2 & 1 \\
    0 & \alpha_3 & 1 
  \end{pmatrix}
  ,
  \quad
  \mathrm{JV} = \begin{pmatrix}
    0 & \alpha_1 & 0 \\
    0 & \alpha_4 & 0 \\
    0 & \alpha_5 & 0 
  \end{pmatrix},
\end{align}
where $\alpha_4$ and $\alpha_5$ are parameters corresponding to the derivatives of $\alpha_2$ and $\alpha_3$, respectively.
Although the Jacobian matrix JD is indeed singular due to $\alpha_2 = \alpha_3$, the $\sigma \nu$-method halts at this point as the method regards $\alpha_2$ and $\alpha_3$ as independent.
This failure originates from the elimination of matrices involving the independent parameter $\alpha_1$.
We have confirmed that the implementation in Mathematica actually fails on this DAE.

Our algorithm is applied to the same DAE~\eqref{def:nonlinear_dae} as follows.
Let
\begin{align}
  A(s) = \begin{pmatrix}
    s   & \alpha &   \\
    s+1 &        & 1 \\
    s   &        & 1 \\
  \end{pmatrix},
\end{align}
where $\alpha$ is an independent parameter representing $\partial g/\partial x_2$.
As described in \cref{sec:ms_alg}, the MS-algorithm is applicable to the nonlinear DAE~\eqref{def:nonlinear_dae} if $A(s)$ is upper-tight.
The tight coefficient matrix corresponding to a dual optimal solution $p = \prn{0, 0, 0}$ and $q = \prn{1, 0, 0}$ is
\begin{align}
  A^\# = \begin{pmatrix}
    1 & \alpha &   \\
    1 &        & 1 \\
    1 &        & 1 \\
  \end{pmatrix},
\end{align}
which is singular.
Thus we need to modify the matrix.
By the same logic as the discussion in \cref{sec:example_2}, we can regard $A(s)$ as an LM-polynomial matrix $A(s) = \binom{T(s)}{Q(s)}$, where $T(s)$ corresponds to the first row and $Q(s)$ corresponds to the other two ones in $A(s)$.
Then our algorithm modifies $A(s)$ to
\begin{align}
  A'(s) = \begin{pmatrix}
    s & \alpha &   \\
    1 &        &   \\
    s &        & 1 \\
  \end{pmatrix},
\end{align}
which is upper-tight (we omit the detail of this modification).
Using an optimal solution $p' = \prn{0, 1, 0}$ and $q' = \prn{1, 0, 0}$ of $\D{A'}$, the MS-algorithm obtains a purely algebraic equation
\begin{align}
  \left\{\begin{alignedat}{4}
    &z^{[1]}_1 + g(x_2)  &&= f_1(t), \\
    &x_1 &&= f_2(t) - f_3(t), \\
    &z^{[1]}_1 &&= \dot{f}_2(t) - \dot{f}_3(t), \\
    &z^{[1]}_1 + x_3 &&= f_3(t),
  \end{alignedat}\right.
\end{align}
where $z^{[1]}_1$ is a dummy variable corresponding to $\dot{x}_1$.

This example shows that our algorithm works for a DAE to which the existing index reduction algorithm cannot be applied.
Our algorithm is expected to rarely cause cancellations between nonlinear terms as it does not perform the row operations involving independent parameters.
In particular, our algorithm can be applied to nonlinear DAEs in which cancellations occur only between linear terms like the transistor amplifier DAE in~\cite{Mazzia2008}; such DAEs often appear in practice.
Therefore, although the application to nonlinear DAEs remains at the stage of a heuristic, it is anticipated that the proposed method can be useful for index reduction of nonlinear DAEs.

\section{Conclusion}
\label{sec:conclusion}

In this paper, we have proposed an index reduction algorithm for linear DAEs whose coefficient matrices are mixed matrices.
The proposed method detects numerical cancellations between accurate constants, and transforms a DAE into an equivalent DAE to which the MS-algorithm is applicable.
Our algorithm uses combinatorial algorithms on graphs and matroids, based on the combinatorial relaxation framework.
We have also developed a faster algorithm for DAEs whose coefficient matrices are dimensionally consistent.
In addition, we have confirmed through numerical experiments that our algorithm runs sufficiently faster than the theoretical guarantee for large scale DAEs, and modifies DAEs preserving physical meanings of dynamical systems.
Our algorithms can also be applied to nonlinear DAEs by regarding nonlinear terms as independent parameters.
Numerical experiments on nonlinear DAEs are left for further investigation.

\section*{Acknowledgments}
We thank anonymous referees for helpful suggestions and comments.
This work was supported in part by JST CREST, Grant Number JPMJCR14D2, Japan.
The second author's research was supported by Grant-in-Aid for JSPS Research Fellow, Grant Number JP18J22141, Japan.

\bibliographystyle{bibstyle}
\bibliography{paper}

\begin{thebibliography}{10}

\bibitem{Brenan1996}
K.~E. Brenan, S.~L. Campbell, and L.~R. Petzold.
\newblock {\em {Numerical Solution of Initial-Value Problems in
  Differential-Algebraic Equations}}.
\newblock SIAM, Philadelphia, 1996.

\bibitem{Chowdhry2004}
S.~Chowdhry, H.~Krendl, and A.~A. Linninger.
\newblock {Symbolic numeric index analysis algorithm for differential algebraic
  equations}.
\newblock {\em Industrial \& Engineering Chemistry Research},
  43(14):3886--3894, 2004.

\bibitem{Cunningham1986}
W.~H. Cunningham.
\newblock {Improved bounds for matroid partition and intersection algorithms}.
\newblock {\em SIAM Journal on Computing}, 15(4):948--957, 1986.

\bibitem{Edmonds1968}
J.~Edmonds.
\newblock {Matroid partition}.
\newblock In G.~B. Dantzig and A.~F. Veinott, editors, {\em Mathematics of the
  Decision Sciences: Part I}, volume~11 of {\em Lectures in Applied
  Mathematics}, pp. 335--345. AMS, Providence, RI, 1968.

\bibitem{Gear1988}
C.~W. Gear.
\newblock {Differential-algebraic equation index transformations}.
\newblock {\em SIAM Journal on Scientific and Statistical Computing}, 9:39--47,
  1988.

\bibitem{Iwata2003}
S.~Iwata.
\newblock {Computing the maximum degree of minors in matrix pencils via
  combinatorial relaxation}.
\newblock {\em Algorithmica}, 36(4):331--341, 2003.

\bibitem{Iwata2001}
S.~Iwata and K.~Murota.
\newblock {Combinatorial relaxation algorithm for mixed polynomial matrices}.
\newblock {\em Mathematical Programming}, 90(2):353--371, 2001.

\bibitem{Iwata2013}
S.~Iwata and M.~Takamatsu.
\newblock {Computing the maximum degree of minors in mixed polynomial matrices
  via combinatorial relaxation}.
\newblock {\em Algorithmica}, 66(2):346--368, 2013.

\bibitem{Iwata2018a}
S.~Iwata and M.~Takamatsu.
\newblock {Index reduction via unimodular transformations}.
\newblock {\em SIAM Journal on Matrix Analysis and Applications},
  39(3):1135--1151, 2018.

\bibitem{Korte2008}
B.~Korte and J.~Vygen.
\newblock {\em {Combinatorial Optimization}}, volume~21 of {\em Algorithms and
  Combinatorics}.
\newblock Springer-Verlag, Berlin Heidelberg, 4th. ed., 2008.

\bibitem{Kuhn1955}
H.~W. Kuhn.
\newblock {The Hungarian method for the assignment problem}.
\newblock {\em Naval Research Logistics Quarterly}, 2:83--97, 1955.

\bibitem{Gall2014}
F.~{Le Gall}.
\newblock {Powers of tensors and fast matrix multiplication}.
\newblock In {\em Proceedings of the 39th International Symposium on Symbolic
  and Algebraic Computation (ISSAC '14)}, pp. 296--303, New York, NY, 2014.
  ACM.

\bibitem{Mattsson1993}
S.~E. Mattsson and G.~S{\"{o}}derlind.
\newblock {Index reduction in differential-algebraic equations using dummy
  derivatives}.
\newblock {\em SIAM Journal on Scientific Computing}, 14(3):677--692, 1993.

\bibitem{Mazzia2008}
F.~Mazzia and C.~Magherini.
\newblock {Test set for initial value problem solvers}.
\newblock Technical report, Department of Mathematics, University of Bari,
  Bari, 2008.
\newblock Retrieved June 1, 2019 from
  \url{https://archimede.dm.uniba.it/~testset/}.

\bibitem{Murota1985b}
K.~Murota.
\newblock {Use of the concept of physical dimensions in the structural approach
  to systems analysis}.
\newblock {\em Japan Journal of Applied Mathematics}, 2(2):471--494, 1985.

\bibitem{Murota1990}
K.~Murota.
\newblock {Computing Puiseux-series solutions to determinantal equations via
  combinatorial relaxation}.
\newblock {\em SIAM Journal on Computing}, 19(6):1132--1161, 1990.

\bibitem{Murota1993}
K.~Murota.
\newblock {Mixed matrices: irreducibility and decomposition}.
\newblock In R.~A. Brualdi, S.~Friedland, and V.~Klee, editors, {\em
  Combinatorial and Graph-Theoretical Problems in Linear Algebra}, volume~50 of
  {\em The IMA Volumes in Mathematics and its Applications}, pp. 39--71.
  Springer, New York, NY, 1993.

\bibitem{Murota1995b}
K.~Murota.
\newblock {Combinatorial relaxation algorithm for the maximum degree of
  subdeterminants: Computing Smith-McMillan form at infinity and structural
  indices in Kronecker form}.
\newblock {\em Applicable Algebra in Engineering, Communication and Computing},
  6(4--5):251--273, 1995.

\bibitem{Murota1995a}
K.~Murota.
\newblock {Computing the degree of determinants via combinatorial relaxation}.
\newblock {\em SIAM Journal on Computing}, 24(4):765--796, 1995.

\bibitem{Murota1998}
K.~Murota.
\newblock {On the degree of mixed polynomial matrices}.
\newblock {\em SIAM Journal on Matrix Analysis and Applications},
  20(1):196--227, 1998.

\bibitem{Murota2000}
K.~Murota.
\newblock {\em {Matrices and Matroids for Systems Analysis}}, volume~20 of {\em
  Algorithms and Combinatorics}.
\newblock Springer-Verlag, Berlin Heidelberg, 2010.

\bibitem{Murota1985a}
K.~Murota and M.~Iri.
\newblock {Structural solvability of systems of equations ---a mathematical
  formulation for distinguishing accurate and inaccurate numbers in structural
  analysis of systems---}.
\newblock {\em Japan Journal of Applied Mathematics}, 2:247--271, 1985.

\bibitem{Murota1987}
K.~Murota, M.~Iri, and M.~Nakamura.
\newblock {Combinatorial canonical form of layered mixed matrices and its
  application to block-triangularization of systems of linear/nonlinear
  equations}.
\newblock {\em SIAM Journal on Algebraic and Discrete Methods}, 8(1):123--149,
  1987.

\bibitem{Oki2019}
T.~Oki.
\newblock {Improved structural methods for nonlinear differential-algebraic
  equations via combinatorial relaxation}.
\newblock In {\em Proceedings of the 44th International Symposium on Symbolic
  and Algebraic Computation (ISSAC '19)}, New York, NY, 2019. ACM.
\newblock To appear.

\bibitem{Oxley2011}
J.~G. Oxley.
\newblock {\em {Matroid Theory}}.
\newblock Oxford Graduate Texts in Mathematics. Oxford University Press, New
  York, NY, 2nd. ed., 2011.

\bibitem{Pantelides1988}
C.~C. Pantelides.
\newblock {The consistent initialization of differential-algebraic systems}.
\newblock {\em SIAM Journal on Scientific and Statistical Computing},
  9(2):213--231, 1988.

\bibitem{Pryce2001}
J.~D. Pryce.
\newblock {A simple structural analysis method for DAEs}.
\newblock {\em BIT Numerical Mathematics}, 41(2):364--394, 2001.

\bibitem{Sato2015}
S.~Sato.
\newblock {Combinatorial relaxation algorithm for the entire sequence of the
  maximum degree of minors in mixed polynomial matrices}.
\newblock {\em JSIAM Letters}, 7:49--52, 2015.

\bibitem{Shi2004}
C.~Shi.
\newblock {\em {Linear Differential-Algebraic Equations of Higher-Order and the
  Regularity or Singularity of Matrix Polynomials}}.
\newblock PhD thesis, Technische Universit\"{a}t, Berlin, 2004.

\bibitem{Tan2017}
G.~Tan, N.~S. Nedialkov, and J.~D. Pryce.
\newblock {Conversion methods for improving structural analysis of
  differential-algebraic equation systems}.
\newblock {\em BIT Numerical Mathematics}, 57(3):845--865, 2017.

\bibitem{Unger1995}
J.~Unger, A.~Kr\"{o}ner, and W.~Marquardt.
\newblock {Structural analysis of differential-algebraic equation systems ---
  theory and applications}.
\newblock {\em Computers and Chemical Engineering}, 19(8):867--882, 1995.

\bibitem{mathematica}
{Wolfram Research, Inc.}
\newblock {Numerical Solution of Differential-Algebraic Equations --- Wolfram
  Language Documentation}, 2017.
\newblock Retrieved September 27, 2017 from
  \url{http://reference.wolfram.com/language/tutorial/NDSolveDAE.html}.

\bibitem{Wu2013}
X.~Wu, Y.~Zeng, and J.~Cao.
\newblock {The application of the combinatorial relaxation theory on the
  structural index reduction of DAE}.
\newblock In {\em Proceedings of the 12th International Symposium on
  Distributed Computing and Applications to Business, Engineering \& Science},
  pp. 162--166. IEEE, 2013.

\end{thebibliography}

\end{document}